\documentclass[a4paper,12pt]{article}
\usepackage[latin1]{inputenc}
\usepackage[T1]{fontenc}
\usepackage[a4paper,lmargin=4.1cm,textwidth=12.8cm,tmargin=4.9cm,textheight=18.5cm]{geometry}
\usepackage[all,2cell,v2]{xy}
\usepackage[all]{xy}
\UseAllTwocells
\usepackage[square,numbers,sort]{natbib}
\usepackage{ifthen}
\usepackage{mathbbol}
\usepackage{savesym}
\usepackage{amsmath}
\usepackage{amssymb}
\savesymbol{iint}
\restoresymbol{TXF}{iint}
\usepackage[thmmarks,standard]{ntheorem}
\usepackage{microtype}
\usepackage{stmaryrd}
\newboolean{PourEditeur}
\setboolean{PourEditeur}{false}
\ifthenelse{\boolean{PourEditeur}}{%
}{%
  \usepackage{color}
  \definecolor{violet}{rgb}{ .615, .122, 1.0}%
  \usepackage[colorlinks,linkcolor=violet,citecolor=violet,urlcolor=violet]{hyperref}
}
\usepackage{cleveref}

\usepackage[toc,page]{appendix}

\theoremstyle{plain}

\newtheorem*{Hypothesis}{Hypotheses} 

\crefformat{result}{result~#2#1#3}
\crefformat{proposition}{proposition~#2#1#3}
\crefformat{remark}{remark~#2#1#3}
\newcommand*{\G}{\ensuremath{\mathbb{G}\text{r}}}

\newcommand*{\UG}{\ensuremath{\omega\text{-}\mathbb{G}\text{R}}}
\newcommand*{\OG}{\ensuremath{\omega\text{-}\G}}
\newcommand*{\C}{\ensuremath{\mathbb{C}AT}}
\newcommand*{\TC}{\ensuremath{\mathbb{T}\text{-}\mathbb{C}at}}
\newcommand*{\TG}{\ensuremath{\mathbb{T}\text{-}\mathbb{G}\text{r}}}
\newcommand*{\TGr}{\ensuremath{\mathbb{T}\text{-}\mathbb{G}\text{rr}}}

\newcommand*{\T}{\ensuremath{\mathbb{T}}}
\newcommand*{\R}{\ensuremath{\mathbb{R}}}
\newcommand*{\GC}{\ensuremath{\mathbb{G}\mathbb{C}AT}}
\newcommand*{\MGC}{\ensuremath{\mathbb{M}\mathbb{G}\mathbb{C}AT}}
\newcommand*{\AMGC}{\ensuremath{\mathbb{A}\mathbb{M}\mathbb{G}\mathbb{C}AT}}

\DeclareMathOperator{\Span}{Span}

\title{$\omega$-Operads of Coendomorphisms for Higher Structures}
\author{Camell Kachour}
\begin{document}
\maketitle
\begin{abstract}
It is well known that strict $\omega$-categories, strict $\omega$-functors, strict natural $\omega$-transformations, and so on, 
form a strict $\omega$-category. A similar property for weak $\omega$-categories is one of the main hypotheses  
in higher category theory in the globular setting. In this paper we show that there is a natural  globular $\omega$-operad
which acts on the globular set of weak $\omega$-categories, weak $\omega$-functors, weak natural $\omega$-transformations, and so on.
Thus to prove the hypothesis it remains to prove that this $\omega$-operad is contractible in Batanin's sense.    
To construct such an $\omega$-operad we introduce more general technology and suggest a definition of $\omega$-operad with the \textit{fractal property}.
If an $\omega$-operad $B^{0}_{P}$ has this property then one can define a globular set of all higher $B^{0}_{P}$-transformations and, moreover, this  
globular set has a $B^{0}_{P}$-algebra structure. 
\end{abstract}
 
\begin{minipage}{118mm}{\small
}\end{minipage}

\tableofcontents
\vspace{1cm}

\section*{Introduction}

The algebraic model of weak higher transformations was undertaken for the first time in 
\cite{kamelkachour:defalg, kach:nscellsfinal} 
with respectively the \textit{Penon approach} and the \textit{Batanin approach}. However Andr{\'e} Joyal has pointed out to us 
that the $2$-coloured $\omega$-operads for the higher transformations that we built in \cite{kach:nscellsfinal} have two 
many coherence cells. In this paper we propose a new approach of contractibilty for coloured $\omega$-operads, which agrees
with contractibility in the sense of Batanin (\cite{bat:monglob}) for monochromatic $\omega$-operads, and which corrects the
counterexample that Andr{\'e} Joyal showed us.  

Overall, this paper is devoted to describing, up to a precise \textit{contractibility hypothesis} 
(see section \ref{Contractibility_Hypothesis}), 
the first operadic approach to the weak $\omega$-category of weak $\omega$-categories
Other approaches have been proposed: For example Michael Makkai in \cite{makkai} has described the weak $\omega$-category of weak 
$\omega$-categories by using a \textit{multitopic} approach, and a \textit{simplicial} approach of the weak 
$\omega$-category of weak $\omega$-categories is describe by Jacob Lurie in 
\cite{LurieTopos}, in the context of the $(\infty,1)$-categories. We believe that our operadic approach
 has the definitive advantage of being very explicit compared to other approaches (see also 
 \cite{JoyalTierney, RezkCartesian,GarnerHomoCat, SimpsonHomoCat}). . 
Such a higher category theory can then be seen 
as the right generalisation of strict $\omega$-category theory, where strict $\omega$-categories,
strict $\omega$-functors, and all higher strict transformations organised themselves into a strict
$\omega$-category. This strict $\omega$-category can be seen  as a sort \textit{fractal object}, in
the sense that its objects are themselves strict $\omega$-categories. 
In \cite{batanin02:_penon_method_of_weaken_algeb_struc} the author has shown that weak $\omega$-categories of Penon are algebras of 
an $\omega$-operad, the \textit{$\omega$-operad of Penon}, and thus by analogy, it should be possible to describe higher transformations in 
\cite{kamelkachour:defalg} as algebras for specifics coloured $\omega$-operads. So this article can be used also as a way to see how to build the 
weak $\omega$-category of weak $\omega$-categories, with the Penon approach to weak $\omega$-categories. Our techniques also work very well for Leinster's version of weak $\omega$-categories which is a slight modification of the original Batanin's approach.

The main technical difficulty was to find the most natural way to build this algebraic model of the weak 
$\omega$-category of 
the weak $\omega$-categories. We believe that 
the way we describe here, up to the \textit{contractibilty hypothesis} in the section 
\ref{Contractibility_Hypothesis}, is the achievement of this goal. The direction that we propose is not only very natural but also it allows us to see quickly how to build it 
(see the section \ref{The_Weak_omega_category_of_the_Weak_omega_categories}). 

More precisely, starting from the coglobular complex of colored $\omega$-operads $B^{\bullet}_{C}$ built in \cite{kach:nscellsfinal} (where $B^{0}_{C}$ is the Batanin's $\omega$-operad for weak $\omega$-categories) we construct its $\omega$-operad of coendomorphisms $Coend(B^{\bullet}_{C})$
(which we called the \textit{violet operad} for reason we explain later) and show that this $\omega$-operad acts on the globular set of weak $\omega$-categories, 
weak $\omega$-functors, etc. We conjecture that the \textit{violet operad}  is contractible (see \ref{Contractibility_Hypothesis}). This conjecture implies immediately  that weak $\omega$-categories, weak $\omega$-functors etc. form  a weak $\omega$-category. We provide some evidence that our contractibility hypothesis is correct but a full proof of  it requires a development of a homotopy theory (and in particular a theory of homotopy colimits) of colored $\omega$-operads. This will be the subject of our future work.

Contractibility is a specific structure
of the $\omega$-operad $B^{0}_{C}$ of Batanin, but if we conceptualise this property and the technology which allows us to see the way
the weak $\omega$-category of weak $\omega$-categories is built, we can describe many kinds of \textit{higher structures} with similar 
\textit{fractal phenomena}, by using the same technology of $\omega$-operads of coendomorphisms. 
Let us be more precise: We start with a basic data, which is a coglobular complex of 
$\omega$-operads $B^{\bullet}_{P}$ in $\TC_c$ equipped with a structure $P$, where $\TC_c$ is the category of $\T$-categories over 
constant $\omega$-graphs, or in other words, the category of coloured $\omega$-operad over constant $\omega$-graphs 
(see section \ref{T-Cat} and section \ref{Some_standard_diagrams_of_the_omega-transformations}). 
We say that the first $\omega$-operad $B^{0}_{P}$ (the "$0$-step") of this coglobular complex $B^{\bullet}_{P}$ has the \textit{fractal property}, 
if there is a morphism of $\omega$-operads between $B^{0}_{P}$ and the corresponding $\omega$-operad of coendomorphisms $Coend(B^{\bullet}_{P})$  
associated to $B^{\bullet}_{P}$. If it happens then all algebras for all $\omega$-operads $B^{n}_{P}$ ($n\in\mathbb{N}$) organise into a single algebra of 
$B^{0}_{P}$. If $P=S_{u}$, where $P$ means \textit{strictly contractible with contractible units} (see the section \ref{Weak_contractibility_and_strict_contractibility} 
and the section \ref{The_Strict_omega_category_of_the_Strict_omega_categories}), then we obtain the indigo operad 
$Coend(B^{\bullet}_{S_{u}})$, and up to a precise \textit{contractibility hypothethis} (see section \ref{Contractibility_Hypothesis}), 
we can describe the strict $\omega$-category of the strict $\omega$-categories with the same technology as we do for the weak case. 
In this article we also build two other coglobular complex $B^{\bullet}_{P}$ of $\omega$-operad which corresponding 
$\omega$-operad $B^{0}_{P}$ have the fractal property, without requiring any hypotheses 
(see section \ref{Examples_of_Standard_Actions}). 

The main ideas of this article were exposed for the first time in September 2010, in the Australian Category Seminar at Macquarie University
 \cite{kach3:redmacq}.

The plan of this article is as follow :

In the first section (see \ref{Batanin_theory_of_omega_Operads}), we summarise Batanin's theory of 
$\omega$-operads (see \cite{bat:monglob})
with the goal to extract the corollary (see \ref{le-corollaire}) which is a central result for our article, and this corollary is just a consequence of proposition 7.2 in \cite{bat:monglob}. A lot of material which surrounds the corollary \ref{le-corollaire} is described in \cite{bat:monglob}: Globular categories, globular functors, 
 monoidal globular categories (called $MG$-categories), monoidal globular functors (called $MG$-functors), augmented monoidal globular categories 
 (called $AMG$-categories), globular objects of a globular category, etc. However we expose these concepts in a more modern approach, which 
 essentially follows the work of \cite{weber:pseudo}. Then we explain in detail the two most important
 $MG$-categories for Batanin's theory of $\omega$-operads: The $MG$-category $\T\text{ree}$ of trees and the $MG$-category 
 $\mathbb{S}pan$ of spans in $Set$ (see \ref{Main_examples_of_Globular_Monoidal_Categories}), which are also described with a modern approach in the works \cite{batross:multi,bat:eckmann,street-petit-topos,mark-topos}. In \ref{T-Cat} we briefly describe $\T$-categories, where 
 $\T$ is the monad of the strict $\omega$-categories on $\omega$-graphs. $\T$-categories are important for this article because
 for us an $\omega$-operad in the sense of Batanin is a $\T$-category over the terminal $\omega$-graph (see the section 
 \ref{omega-Operads_of_Endomorphism_and_Coendomorphism}).  
  
In the second section (see \ref{Some_standard_diagrams_of_the_omega-transformations}) we state the main result of the article :
 By using the corollary \ref{le-corollaire} of the previous section, for each coglobular object $W^{\bullet}$ in $\TC_c$ we associate
 its \textit{standard action in $\TC_1$} which roughly speaking is a diagram in $\TC_1$ built with two morphisms of $\omega$-operads.
 In particular each coglobular object $W^{\bullet}$ shows us two important $\omega$-operads : The $\omega$-operad $W^{0}$ (the 
 "$0$"-step of the coglobular complex $W^{\bullet}$), and the associated $\omega$-operad of coendomorphism 
 $Coend(W):=(HOM(W^{n},W^{t}))_{n\in \mathbb{N}, t\in\mathbb{T}ree}$. The $\omega$-operad $W^{0}$ is \textit{fractal}
 if we can build a morphism of $\omega$-operads between it and $Coend(W)$. 
 Then we give the application of these technology to the coglobular complex $C^{\bullet}$ in the category $\TG_{p,c}$ of pointed $\T$-graphs 
 over constant $\omega$-graphs. For example denote by $B^{0}_C$ the $\omega$-operad of Batanin for weak $\omega$-categories, and by
 $Coend(B^{\bullet}_C)$ the $\omega$-operad of coendomorphisms of the coglobular complex $B^{\bullet}_C$ in $\TC_c$ freely 
 generated by $C^{\bullet}$. If $B^{\bullet}_C$ is fractal then there is an action of it on the globular complex in $SET$ of the weak
 higher transformations, which show that the weak $\omega$-category of weak $\omega$-categories exists in a completely
 $\omega$-operadic setting. It shows also that it is a weak $\omega$-category in the sense of Batanin. 
 
The third section is devoted first to describing the coglobular complex $B^{\bullet}_{S_{u}}$ in $\TC_c$ of  strict higher
transformations, and the coglobular complex $B^{\bullet}_{C}$ in $\TC_c$ of weak higher transformations. In particular
we propose a new approach to contractibility which corrects a counterexample that Andr{\'e} Joyal constructed to
our first approach to contractibility as in the article \cite{kach:nscellsfinal} for weak higher transformations. The key
points of this new approach is to bring to light some remarkable cells that we call \textit{root cells}, and to take account of a specific
property, the \textit{loop property}, that these \textit{root cells} must follow for contractibility. Then we state the following
hypothesis : For each tree $t$, the coloured $\omega$-operad $B^{t}_{S_{u}}$ is strictly contractible and has contractible units,
and the coloured $\omega$-operad $B^{t}_{C}$ is contractible. If we accept this hypothesis then it is possible to build a composition system
for each $\omega$-operad of coendomorphism $Coend(B^{\bullet}_{S_{u}})$ and $Coend(B^{\bullet}_C)$, and also to show that the $\omega$-operads
 $B^{0}_{S_{u}}$ and $B^{0}_{C}$ are fractal. It thus show that the strict $\omega$-category of strict 
 $\omega$-categories exists in a completely
 $\omega$-operadic setting, the weak $\omega$-category of weak $\omega$-categories exists in a completely $\omega$-operadic setting,
 and this facts are proved by using the same technology related to the \textit{standard action in $\TC_1$}.
 
The fourth section gives two examples of $\omega$-operads having the fractal property : It is easy to show that the $\omega$-operad
$B^{0}_{Id}$ of $\omega$-magmas and the $\omega$-operad $B^{0}_{Id_{u}}$ of reflexive $\omega$-magmas, 
both have the fractal property. This proves the existence of the $\omega$-magma of $\omega$-magmas, and the reflexive $\omega$-magma of 
reflexive $\omega$-magmas, by using the same technology related to the \textit{standard action in $\TC_1$}.

{\bf Acknowledgement.} 
I am grateful to Michael Batanin and to Ross Street for their mathematical support and encouragement. I am grateful to
Steve Lack, Mark Weber, who explained some technical points that I was not able to understand by myself. I am 
grateful to Andr{\'e} Joyal who helped me to improve my $\omega$-operadic approach to weak higher transformations.
I am also grateful to Richard Garner who shared with me his point of view on the small argument object, and to Cl{\'e}mens Berger, 
Denis-Charles Cisinski, and R{\'e}my Tuy{\'e}ras for many discussions about abstract homotopy theory. Finally I am grateful to some of my friends who support me many times with skype: Samir Berrichi, 
Bushr\^{a} Belguellil, Abdou Azzouz, Nizar Slimani, and Jean-Pierre Ledru, and other friends who support 
me in Australia: Chris McMillan, Estelle Helene Borrey, Edwin Nelson, and Omar Kali.

I dedicate this work to Michael Batanin.

  \vspace*{1cm}   

  \section{Batanin's theory of $\omega$-Operads}
  \label{Batanin_theory_of_omega_Operads}
  
 Thoughout this paper, if $\mathbb{C}$ is a category then $\mathbb{C}(0)$ is the class of its
objects (but we often omit "$0$" when there is no confusion) and $\mathbb{C}(1)$ is the class of its morphisms. The symbol $:=$ means
"by definition is". Also $Set$ denotes the category of sets, and $SET$ denotes the category of large sets (for instance the proper class
of ordinals is an object of $SET$, but not in $Set$). Similarly $\mathbb{C}at$ denotes the $2$-category
 of small categories, and $\C$ denotes the $2$-category of categories. 
  
  The theory of $\omega$-operads was developed for the first time by Michael Batanin in his seminal article \cite{bat:monglob}. More precisely,
  he produced a theory of $\omega$-operads in the general context of his monoidal globular categories. 
       
  In this chapter we summarise the general approach of the theory of $\omega$-operads of Michael Batanin, 
  because it is in this general approach that the important corollary \ref{le-corollaire} was formulated. This corollary is the key result to developing 
  the main technology of this article : It is a result about the existence of the $\omega$-operad of coendomorphisms which, as we will see, plays
  an important role for many kinds of \textit{higher structure}. A \textit{higher structure} means for us a structure based on $\omega$-graphs. For instance, 
  $\omega$-magmas are basic example of such higher structure, but we will consider also reflexive 
  $\omega$-magmas as an other kind of higher structure, and 
  also other more complex higher structures like the weak $\omega$-categories.
        
 \subsection{$MG$-categories and $AMG$-categories}
 \label{MG-categories}
 A lot of material which surrounds corollary \ref{le-corollaire} is described in \cite{bat:monglob}: Globular categories, globular functors, 
 monoidal globular categories (called $MG$-categories), monoidal globular functors (called $MG$-functors), augmented monoidal globular categories 
 (called $AMG$-categories), globular objects of a globular category, etc. 

   \begin{definition}  
   \label{def-glob-category}
 The globe category $\mathbb{G}$ is defined as following: For each $n\in\mathbb{N}$, objects of $\mathbb{G}$ are formal objects $\bar{n}$. Morphisms of $\mathbb{G}$ 
  are generated by the (formal) cosource and cotarget $\xymatrix{\bar{n}\ar[r]<+2pt>^{s^{n+1}_{n}}\ar[r]<-2pt>_{t^{n+1}_{n}}&\overline{n+1}}$ 
  such that we have the relations $s^{n+1}_{n}s^{n}_{n-1}=s^{n+1}_{n}t^{n}_{n-1}$ and $t^{n+1}_{n}t^{n}_{n-1}=t^{n+1}_{n}s^{n}_{n-1}$. 
  For each $0\leqslant p<n$, we put $s^{n}_{p}:=s^{n}_{n-1}\circ s^{n-1}_{n-2}\circ ... \circ s^{p+1}_{p}$ and 
  $t^{n}_{p}:=t^{n}_{n-1}\circ t^{n-1}_{n-2}\circ ... \circ t^{p+1}_{p}$.  
 \end{definition} 
 \begin{definition}
 Starting with the globe category $\mathbb{G}$ above, 
 we build the reflexive globe category $\mathbb{G}_{\text{r}}$ as follow : For each $n\in\mathbb{N}$ we add in $\mathbb{G}$ the formal morphism
 $\xymatrix{\overline{n+1}\ar[rr]^{1^{n}_{n+1}}&&\bar{n}}$ such that 
 $1^{n}_{n+1}\circ s^{n+1}_{n}=1^{n}_{n+1}\circ t^{n+1}_{n}=1_{\bar{n}}$. For each $0\leqslant p<n$, we put
 $1^{p}_{n}:=1^{p}_{p+1}\circ 1^{p+1}_{p+2}\circ ... \circ 1^{n-1}_{n}$. 
 \end{definition} 
    
 The category of $\omega$-graphs is the category of presheaves $\OG:=[\mathbb{G}^{op};Set]$ (which is also called the category of globular sets in the 
 literature; see for example \cite{bat:monglob}), the category of large $\omega$-graphs is the category of presheaves 
 $\UG:=[\mathbb{G}^{op};SET]$, and the $2$-category of globular categories is the $2$-category of prestacks $\GC:=[\mathbb{G}^{op};\C]$.
  
\begin{definition}
Consider the terminal globular category $1$ and a globular category $\mathcal{C}$.
 A globular object $(W,\mathcal{C})$ in
 $\mathcal{C}$ is a morphism $\xymatrix{1\ar[r]^{W}&\mathcal{C}}$ in $\GC$.
\end{definition}
 
 Let us put $\OG\text{r}:=[\mathbb{G}^{op}_{\text{r}};Set]$, the category of the reflexive $\omega$-graphs (see \cite{penon1999}). We have the adjunction
  \[\xymatrix{U\dashv R: \OG\text{r}\ar[rr]&&\OG}\]
 and we call $(\R,\eta,\mu)$ the generated monad whose algebras are reflexive $\omega$-graphs.
 Objects of $\OG\text{r}$ are usually denoted by $(G,(1^{p}_{n})_{0\leqslant p<n}))$, where the operations 
 $(1^{p}_{n})_{0\leqslant p<n}$ form a chosen reflexive structure on the $\omega$-graph $G$.  
     
  Let us denote by $\omega\text{-}\mathbb{C}at$ the category of strict $\omega$-categories. The forgetful functor 
  $\xymatrix{\omega\text{-}\mathbb{C}at\ar[r]^{U}&\OG}$, which associates to any strict $\omega$-category $C$
 its underlying $\omega$-graph $U(C)$, is monadic. The corresponding adjunction generates a cartesian monad $\T$
  which is the monad of strict $\omega$-categories on $\omega$-graphs.

 Consider $\mathbb{C}AT_{Pull}$, the $2$-category of categories with pullbacks, with morphisms functors which preserve 
 these pullbacks, and with $2$-cells natural transformations between these functors. The functor $\mathbb{C}at(-)$ which 
 associates to any object $C$ in $\mathbb{C}AT_{Pull}$ the $2$-category $\mathbb{C}at(C)$ of internal categories in it, is 
 a $2$-functor
   \[\xymatrix{\mathbb{C}AT_{Pull}\ar[rr]^{\mathbb{C}at(-)}&&2\text{-}\mathbb{C}AT}\]
 where here $2\text{-}\mathbb{C}AT$ denote the $2$-category of $2$-categories. Thus for the case of the monad 
 $\T$ on $\OG$ we can associate the $2$-monad $\mathcal{T}=\mathbb{C}at(\T)$ on $\GC$.
  
 \begin{definition}[\cite{weber:pseudo}]
 An $MG$-category is a normal pseudo $\mathcal{T}$-algebra for the 
 $2$-monad $\mathcal{T}$ on $\GC$, $MG$-functors are
 strong $\mathcal{T}$-morphisms, and $MG$-natural transformations are algebra $2$-cells of $\mathcal{T}$.
 These data form the $2$-category $\MGC$ of $MG$-categories. 
 \end{definition} 
 
There is a coherence result in \cite{bat:monglob} that any $MG$-category is equivalent to a strict $MG$-category (a strict $MG$-category
 is just an internal strict $\omega$-category in $\C$). Because of this coherence theorem we will not mention explicitly the coherence
 isomorphisms in the $MG$-categories which can be found in \cite{bat:monglob}. 
 Also, the $2$-category $\MGC$ has a cartesian monoidal structure, which allows us to make the following definition

 \begin{definition}[\cite{weber:pseudo}]
An $AMG$-category is a pseudo monoid in $\MGC$. An $AMG$-functor is an $MG$-functor 
$f : A \longrightarrow A^{\prime}$ equipped with a
strong monoidal structure. An $AMG$-natural transformation $\phi : f \Longrightarrow f^{\prime}$ is an $MG$-natural transformation such that
$\phi$ is a monoidal $2$-cell. These datas form the $2$-category $\AMGC$ of the $AMG$-categories.
\end{definition}

\subsection{Main examples of Monoidal Globular Categories}
    \label{Main_examples_of_Globular_Monoidal_Categories}
Globular categories can be defined also as internal categories in $\OG$ because of the canonical isomorphism
 $\mathbb{C}at(\OG)\simeq [\mathbb{G}^{op};\mathbb{C}at]$. We will use this presentation to define the strict
 $MG$-category of $n$-trees as a discrete internal category
 \[\xymatrix{\T(1)\ar[rrr]&&&\T(1)\ar[lll]<+3pt>\ar[lll]<-3pt>\ar[rrr]<-3pt>\ar[rrr]<+3pt>&&&\T(1)\ar[lll]<+6pt>\ar[lll]\ar[lll]<-6pt>}\]
  This $MG$-category $\T\text{ree}$ has a canonical globular object given by the unit of $\T$ : 
  $\xymatrix{1\ar[r]^{}&\T\text{ree}},\xymatrix{1(n)\ar@{|->}[r]&1(n)}$, where here $1$ denote the terminal globular categories,
  where here $1(n)$\footnote{which is denoted $U_{n}$ in \cite{bat:monglob}} denotes the $n$-linear tree. 
  It is shown in \cite{batross:multi,bat:eckmann} that it has the following universal 
  property : If $\mathcal{C}$ is an $MG$-category and $(\mathcal{C},W)$ is a globular object in it, then there is 
 a unique, up to isomorphism, $MG$-functor $W(-)$ which makes commutative the following triangle
 \[\xymatrix{\T\text{ree}\ar@{.>}[rrdd]^{W(-)}\\\\
 1\ar[uu]^{}\ar[rr]_{W}&&\mathcal{C}}\] 
  
Let us set up the following notation : Tensors of the monoidal globular category of $n$-trees are denoted by symbols $\star^{n}_{p}$
 \[\xymatrix{\star^{n}_{p} : \T\text{ree}_n\underset{\T\text{ree}_p}\times \T\text{ree}_n\ar[rr]&&\T\text{ree}_n}\]
   Also an $n$-tree $t$ can be degenerate if it is of the form $t=Z^{k}_{n}(t')$ where $t'$ is a $k$-tree such that $0\leqslant k<n$. 
In \cite{bat:monglob} the author used the letter "Z" to express the reflexivity of an $MG$-category, but we prefer use the notation "1" to express these 
reflexivities for the specific case of $n$-trees, to emphasis that a degenerate tree $t=1^{k}_{n}(t')$ is also an $n$-cell of the strict $\omega$-category 
$\T(1)$. For example, for the $n$-linear tree $1(n)$, the $(n+1)$-tree $t=1^{n}_{n+1}(1(n))$ of $\T(1)$ is degenerate.
 
 Each $n$-tree $t$ has a unique decomposition
 $$1^{k_{1}}_{i_{1}}(1(k_{1}))\star^{sup(i_{1},i_{2})}_{i_{1}'}1^{k_{2}}_{i_{2}}(1(k_{2}))
 \star^{sup(i_{2},i_{3})}_{i_{2}'} ... \star^{sup(i_{m-1},i_{m})}_{i_{m-1}'}1^{k_{m}}_{i_{m}}(1(k_{m}))$$
  where for each $1\leqslant j\leqslant m-1$, we have $i'_{j}<k_{j+1}\leqslant i_{j+1}$ and $i'_{j}<k_{j}\leqslant i_{j}$, and if $k_{j}=i_{j}$ by convention we put 
  $1^{k_{j}}_{i_{j}}(1(k_{j}))=1(k_{j})$. From this unique decomposition, the $n$-tree $t$ can be represented
  by the matrix of numbers 
  \[ \left( \begin{array}{ccccccccc}
i_{1} &  & i_{2} & . & . & . & i_{m-1} & & i_{m}\\
  & i'_{1} & . & . & . & . & . & i'_{m-1} & \end{array} \right)\]    
   which we call the \textit{Grothendieck notation} for the $n$-tree $t$ 
   (see \cite{grothendieck83:_pursuin_stack,malts:group,Arathese}). 
Many authors gave their own approach to $n$-trees (see for instance \cite{street-petit-topos,joyal:theta,bat:monglob,lein1:oper,Ber00,Cisinsk:Bat}), 
and all these approaches are equivalent. 
 
The second class of important examples of $MG$-category is given by the \textit{Span and Cospan construction}. For each $n\in \mathbb{N}$, consider
 the following formal partialy ordered set $Oct(n)$
 
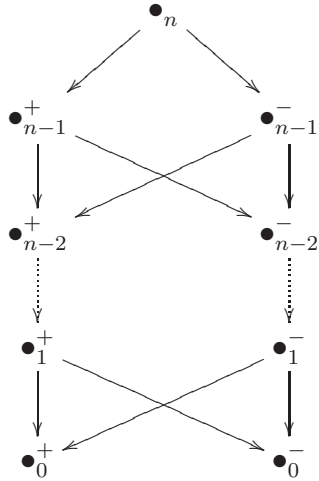
\[\xymatrix{&\bullet^{}_{n}\ar[ld]\ar[rd]\\
  \bullet^{+}_{n-1}\ar[d]\ar[rrd]&&\bullet^{-}_{n-1}\ar[d]\ar[lld]\\
  \bullet^{+}_{n-2}\ar@{.>}[d]&&\bullet^{-}_{n-2}\ar@{.>}[d]\\   
  \bullet^{+}_{1}\ar[d]\ar[rrd]&&\bullet^{-}_{1}\ar[d]\ar[lld]\\ 
  \bullet^{+}_{0}&&\bullet^{-}_{0}}\]
  
Let $Oct^{+}(n-1)$ be the poset obtained from $Oct(n)$ by removing $\bullet^{-}_{n-1}$ and $\bullet^{}_{n}$, and similarly 
Let $Oct^{-}(n-1)$ be the poset obtained from $Oct(n)$ by removing $\bullet^{+}_{n-1}$ and $\bullet^{}_{n}$

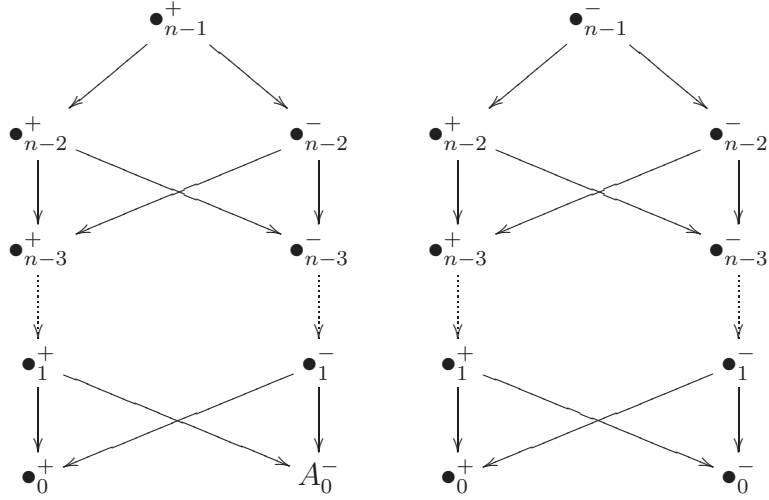
  \[\xymatrix{&\bullet^{+}_{n-1}\ar[ld]\ar[rd]\\
  \bullet^{+}_{n-2}\ar[d]\ar[rrd]&&\bullet^{-}_{n-2}\ar[d]\ar[lld]\\
  \bullet^{+}_{n-3}\ar@{.>}[d]&&\bullet^{-}_{n-3}\ar@{.>}[d]\\   
  \bullet^{+}_{1}\ar[d]\ar[rrd]&&\bullet^{-}_{1}\ar[d]\ar[lld]\\ 
  \bullet^{+}_{0}&&A^{-}_{0}}\qquad \xymatrix{&\bullet^{-}_{n-1}\ar[ld]\ar[rd]\\
  \bullet^{+}_{n-2}\ar[d]\ar[rrd]&&\bullet^{-}_{n-2}\ar[d]\ar[lld]\\
  \bullet^{+}_{n-3}\ar@{.>}[d]&&\bullet^{-}_{n-3}\ar@{.>}[d]\\   
  \bullet^{+}_{1}\ar[d]\ar[rrd]&&\bullet^{-}_{1}\ar[d]\ar[lld]\\ 
  \bullet^{+}_{0}&&\bullet^{-}_{0}}\] 

We obtain the following diagram in $\mathbb{C}at$
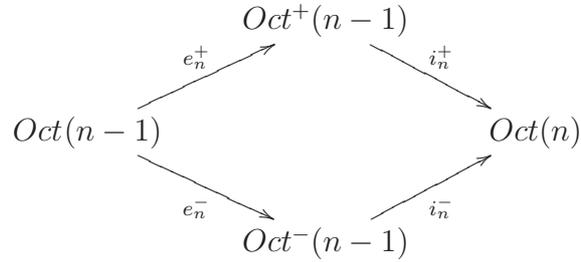
\[\xymatrix{&Oct^{+}(n-1)\ar[rd]^{i^{+}_{n}}\\
  Oct^{}(n-1)\ar[ru]^{e^{+}_{n}}\ar[rd]_{e^{-}_{n}}&&Oct(n)\\
  &Oct^{-}(n-1)\ar[ru]_{i^{-}_{n}}}\]
such that functors $i^{+}_{n}, \, i^{-}_{n}$ are just canonical inclusions, and the functors $e^{+}_{n}$ and 
 $e^{-}_{n}$ are obvious isomorphisms. Put $s^{n}_{n-1}=i^{+}_{n}\circ e^{+}_{n}$ and $t^{n}_{n-1}=i^{-}_{n}\circ e^{-}_{n}$. The family
 of functors $\xymatrix{Oct(n-1)\ar[r]<+2pt>^(.6){s^{n}_{n-1}}\ar[r]<-2pt>_(.6){t^{n}_{n-1}}&Oct(n)}$ ($n\geqslant 1$), defines an object of 
 $\GC$.
 
 Furthermore, 
 for any category $C\in\mathbb{C}AT$, let us call the category of $n$-spans in $C$ the following category of presheaves in $C$:
$Span_{n}(C):=[Oct(n);C]$. The previous functors $s^{n}_{n-1}$ and $t^{n}_{n-1}$ induce a family
 of functors $\xymatrix{Span_{n}(C)\ar[r]<+2pt>^{s^{n}_{n-1}}\ar[r]<-2pt>_{t^{n}_{n-1}}&Span_{n-1}(C)}$ 
 (that we still note by $s^{n}_{n-1}$ and $t^{n}_{n-1}$ because there is no risk of confusion), which defines an object 
 of $\mathbb{G}\mathbb{C}AT$. Dually for any category $C\in\mathbb{C}AT$, let us call the category of presheaves
$Cospan_{n}(C):=[Oct(n)^{op};C]$, the category of $n$-cospans in $C$. The functors $s^{n}_{n-1}$ and $t^{n}_{n-1}$ between the $Oct(n)$, also induce a family
 of functors $\xymatrix{Cospan_{n}(C)\ar[r]<+2pt>^{s^{n}_{n-1}}\ar[r]<-2pt>_{t^{n}_{n-1}}&Cospan_{n-1}(C)}$, 
 which is still an object of $\mathbb{G}\mathbb{C}AT$. These two constructions are functorial and define the Span and Cospan constructions 
 
    \[\xymatrix{\mathbb{C}AT\ar[rr]<+2pt>^{Span}\ar[rr]<-2pt>_{Cospan}&&\mathbb{G}\mathbb{C}AT}.\]
 
  The case of a category $C$ with pullbacks is more interesting for the span construction, because the corresponding globular category $Span(C)$
  is canonically equipped with an $MG$-structure. We have a dual result for categories with pushouts and their cospans (see example 7 of section 3 in \cite{bat:monglob}). 
  For example consider a category $C$ with pushouts and the two $2$-cospans $x$ and $y$ in $C$ ($x$ is the diagram on the left).
  
    \[\xymatrix{&A_{2}\\
  A^{+}_{1}\ar[ru]&&A^{-}_{1}\ar[lu]\\
  A^{+}_{0}\ar[u]\ar[rru]&&A^{-}_{0}\ar[u]\ar[llu]}   
  \qquad \xymatrix{&B_{2}\\
  B^{+}_{1}\ar[ru]&&A^{+}_{1}\ar[lu]\\
  A^{+}_{0}\ar[u]\ar[rru]&&A^{-}_{0}\ar[u]\ar[llu]}\] 
The $1$-cospans $s^{2}_{1}(x)$ and $t^{2}_{1}(y)$ are equal to the following $1$-cospan
               
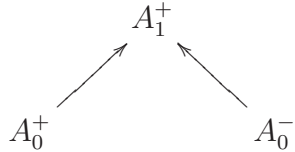
\[\xymatrix{&A^{+}_{1}\\
  A^{+}_{0}\ar[ru]&&A^{-}_{0}\ar[lu]}\]
and $x\otimes^{2}_{1}y$ is given by the following $2$-cospan in $C$.   
        \[\xymatrix{&A_{2}\underset{A^{+}_{1}}\sqcup_{} B_{2}\\
  A^{+}_{1}\ar[ru]&&B^{+}_{1}\ar[lu]\\
  A^{+}_{0}\ar[u]\ar[rru]&&A^{-}_{0}\ar[u]\ar[llu]}\]
 
  Consider $\mathbb{C}AT_{Push}$ the
  category of categories with pushouts and with morphisms functors which preserve these pushouts. Dually consider the underlying category of the $2$-category
  $\mathbb{C}AT_{Pull}$ that we have introduced the section \ref{MG-categories}. We have the following diagram 
  
  \[\xymatrix{\mathbb{C}AT_{Push}\ar[dd]_{(.)^{op}}\ar[rd]^{Cospan}&\\
   &\mathbb{G}\mathbb{M}\mathbb{C}AT  \\
   \mathbb{C}AT_{Pull} \ar[ru]_{Span}   
         }\]   
where $(.)^{op}$ is the basic isomorphism of categories coming from duality and where
    the functors $Cospan$ and $Span$ are easily defined 
    on morphisms by construction, just because functors preserving pushouts gives $MG$-functors
    between their categories of cospans. 
    
\begin{remark}
If $\mathbb{C}AT^{*}_{Pull}$ denotes the category of categories with pullbacks and initial objects, and morphisms functors which preserve pullbacks and initial objects, and 
$\mathbb{C}AT^{*}_{Push}$ denotes the category of categories with pushouts and initial objects, and morphisms functors which preserve pushouts and initial objects, then
we have the following constructions
\[\xymatrix{\mathbb{C}AT^{*}_{Push}\ar[dd]_{(.)^{op}}\ar[rd]^{Cospan}&\\
   &\mathbb{A}\mathbb{M}\mathbb{G}\mathbb{C}AT  \\
   \mathbb{C}AT^{*}_{Pull} \ar[ru]_{Span}   
         }\]   
\end{remark}

Now consider a category $C$ with pushouts and a globular object $(C,W)$ in $Cospan(C)$, which is also a coglobular object in $C$. 
Thanks to the universality of the map $\xymatrix{1\ar[r]^{}&\T\text{ree}}$ above there exist a unique map 
 \[\xymatrix{W(-):\T\text{ree}\ar[rr]&&Cospan(C)}\]
This map $W(-)$ sends each $n$-tree $t$ to a $n$-coglobular object in $C$ :

\[\xymatrix{W(t)=(W^{0}\ar[r]<+2pt>^{\delta^{1}_{0}}\ar[r]<-2pt>_{\kappa^{1}_{0}}
  &W^{\partial^{n-1} t}\ar[r]<+2pt>^{\delta^{1}_{2}}\ar[r]<-2pt>_{\kappa^{2}_{1}}
  &W^{\partial^{n-2} t}\ar@{.>}[r]<+2pt>^{}\ar@{.>}[r]<-2pt>_{}
  &W^{\partial t}\ar[r]<+2pt>^{\delta^{n}_{n-1}}\ar[r]<-2pt>_{\kappa^{n}_{n-1}}
  &W^{t}),}\] 
 where the $\partial^{k} t$ denotes the trunction of the $n$-tree $t$ in the level $k$ ($1\leqslant k\leqslant n-1$). In this $n$-coglobular object $W(t)$, 
 $W^{t}$ denotes the colimit in $C$ of the diagram
 
       \[\xymatrix{ W^{i_{1}}&&  W^{i_{2}}\ar@{.}[r]&W^{i_{m-2}}&&W^{i_{m-1}}\\
    &W^{i_{1}^{\prime}}\ar[lu]^{\kappa^{i_{1}^{\prime}}_{i_{1}}} \ar[ru]_{\delta^{i_{1}^{\prime}}_{i_{2}}}
    \ar@{.}[rrr]
        &&& W^{i_{m-1}^{\prime}}\ar[lu]^{\kappa^{i_{m-1}^{\prime}}_{i_{m-2}}} 
        \ar[ru]_{\delta^{i_{m-1}^{\prime}}_{i_{m-1}}}&}\]
 coming from the Grothendieck presentation of the $n$-tree $t$.

 $Span:=Span(Set)$ is an important $MG$-category. 
  Examples of $n$-spans in $Set$ are given by the \textit{HOM construction} : For each globular categories $\mathcal{C}\in\mathbb{G}\mathbb{C}AT$,
and for each pair of objects $A,B\in \mathcal{C}_{n}$, we associate the following $n$-span $HOM(A,B)$ in $Set$
        \[\xymatrix{&HOM(A,B)_{n}\ar[ld]_{s^{n}_{n-1}}\ar[rd]^{t^{n}_{n-1}}\\
  HOM(A,B)_{n-1}\ar[d]_{s^{n-1}_{n-2}}\ar[rrd]^(.3){t^{n-1}_{n-2}}&&HOM(A,B)_{n-1}\ar[d]^{t^{n-1}_{n-2}}\ar[lld]^(.7){s^{n-1}_{n-2}}\\
  HOM(A,B)_{n-2}\ar@{.>}[d]&&HOM(A,B)_{n-2}\ar@{.>}[d]\\   
  HOM(A,B)_{1}\ar[d]_{s^{1}_{0}}\ar[rrd]^(.3){t^{1}_{0}}&&HOM(A,B)_{1}\ar[d]^{t^{1}_{0}}\ar[lld]^(.7){s^{1}_{0}}\\ 
  HOM(A,B)_{0}&&HOM(A,B)_{0}}\]
which is such that $HOM(A,B)_{n}:=hom_{\mathcal{C}_{n}}(A,B)$, and for all $0\leqslant k<n$, $HOM(A,B)_{k}:=hom_{\mathcal{C}_{k}}(s^{n}_{k}(A),s^{n}_{k}(B))$, where 
$(s^{k+1}_{k})_{0\leqslant k\leqslant n-1}$ and $(t^{k+1}_{k})_{0\leqslant k\leqslant n-1}$, are given by the functor sources and functor targets of the 
globular category $\mathcal{C}$. 

Now consider a category $C$ with pushouts and a globular object $(C,W)$ in $Cospan(C)$.
If $t$ is a $n$-tree we can associate between $W(1(n))$ and $W(t)\in Cospan(C)_{n}$
the $n$-span $HOM(W(1(n)),W(t))$, such that elements of the set $HOM(W(1(n)),W(t))_{n}$ are diagrams of the form

  \[ \xymatrix{W^{n}\ar[rr]^{f_{n}}&&W^{t}\\
     W^{n-1}\ar[u]<+2pt>^{\delta^{n}_{n-1}}\ar[u]<-2pt>_{\kappa^{n}_{n-1}}
       \ar[rr]<+2pt>^{f^{-}_{n-1}}\ar[rr]<-2pt>_{f^{+}_{n-1}}
       &&W^{\partial t}\ar[u]<+2pt>^{\delta^{t}_{\partial t}}\ar[u]<-2pt>_{\kappa^{t}_{\partial t}}\\
       W^{n-(k-1)}\ar@{.>}[u]<+2pt>^{}\ar@{.>}[u]<-2pt>_{}      
        \ar[rr]<+2pt>^{f^{-}_{n-(k-1)}}\ar[rr]<-2pt>_{f^{+}_{n-(k-1)}}&&
       W^{\partial^{k-1}t}\ar@{.>}[u]<+2pt>^{}\ar@{.>}[u]<-2pt>_{}\\    
       W^{n-k}\ar[u]<+2pt>^{\delta^{n-(k-1)}_{n-k}}\ar[u]<-2pt>_{\kappa^{n-(k-1)}_{n-k}}
       \ar[rr]<+2pt>^{f^{-}_{n-k}}\ar[rr]<-2pt>_{f^{+}_{n-k}}
       &&W^{\partial^{k}t}\ar[u]<+2pt>^{\delta^{\partial^{k-1}t}_{\partial^{k}t}}\ar[u]<-2pt>_{\kappa^{\partial^{k-1}t}_{\partial^{k}t}}\\
       W^{1}\ar@{.>}[u]<+2pt>^{}\ar@{.>}[u]<-2pt>_{}      
        \ar[rr]<+2pt>^{f^{-}_{1}}\ar[rr]<-2pt>_{f^{+}_{1}}&&
       W^{\partial^{n-1}t}\ar@{.>}[u]<+2pt>^{}\ar@{.>}[u]<-2pt>_{}\\    
       W^{0}\ar[u]<+2pt>^{\delta^{1}_{0}}\ar[u]<-2pt>_{\kappa^{1}_{0}}
       \ar[rr]<+2pt>^{f^{-}_{0}}\ar[rr]<-2pt>_{f^{+}_{0}}
       &&W^{0}\ar[u]<+2pt>^{\delta^{\partial^{n-1}t}_{0}}\ar[u]<-2pt>_{\kappa^{\partial^{n-1}t}_{0}}       
    }\]  
    
   which commute serially, that is: 
   
 \begin{itemize}
   \item $f_{n}\delta^{n}_{n-1}=\delta^{t}_{\partial t}f^{-}_{n-1}$, $f_{n}\kappa^{n}_{n-1}=\kappa^{t}_{\partial t}f^{+}_{n-1}$  
    \item $\forall 1\leqslant k\leqslant n-1, f^{-}_{n-(k-1)}\delta^{n-(k-1)}_{n-k}=\delta^{\partial^{k-1}t}_{\partial^{k}t}f^{-}_{n-k}$,     
    $f^{-}_{n-(k-1)}\kappa^{n-(k-1)}_{n-k}=\kappa^{\partial^{k-1}t}_{\partial^{k}t}f^{-}_{n-k}$ 
    and  $f^{+}_{n-(k-1)}\delta^{n-(k-1)}_{n-k}=\delta^{\partial^{k-1}t}_{\partial^{k}t}f^{+}_{n-k},     
    f^{+}_{n-(k-1)}\kappa^{n-(k-1)}_{n-k}=\kappa^{\partial^{k-1}t}_{\partial^{k}t}f^{+}_{n-k}$. 
 \end{itemize}
See also paragraph 9.2 in \cite{lein1:oper}. 

\begin{remark}
 Spans in sets can be seen in a conceptual way: In \cite{street-petit-topos}, Ross Street has shown that objects of $Span$ 
 are internal sets in the petit topos $\omega$-$\G$ of $\omega$-graphs, and in \cite{mark-topos} Mark Weber has shown that $Span$ is a discrete opfibration classifier in 
 the $2$-topos $\GC$ of globular categories 
\end{remark}

We can summarise many constructions of this section 
with the following diagram in $2$-$\mathbb{C}AT$

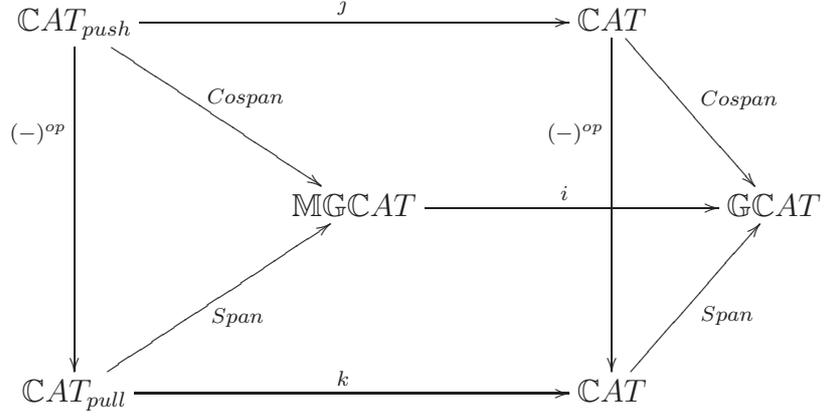
 \[\xymatrix{\mathbb{C}AT_{push}\ar[rrrr]^{j}\ar[dddd]_(.3){(-)^{op}}\ar[rrdd]^{Cospan}&&&&\mathbb{C}AT\ar[dddd]_(.3){(-)^{op}}\ar[rdd]^{Cospan}\\\\
  &&\mathbb{M}\mathbb{G}\mathbb{C}AT\ar[rrr]^{i}&&&\mathbb{G}\mathbb{C}AT\\\\
  \mathbb{C}AT_{pull}\ar[rruu]_{Span}\ar[rrrr]^{k}&&&&\mathbb{C}AT\ar[ruu]_{Span}}\]
Recall that in this section we denote by $1$ the terminal globular category, and in 
definition \ref{def-glob-category} we have denoted the globe category by $\mathbb{G}$. 
\begin{lemma}
We have the following identifications
\begin{itemize}
  \item $(1\downarrow i)$ is the comma category of the globular objects 
$\xymatrix{1\ar[r]^{W}&\mathcal{C}}$ such that $\mathcal{C}\in\mathbb{M}\mathbb{G}\mathbb{C}AT$,\\
 \item $(1\downarrow i\circ Cospan)$ is the comma category of the globular objects 
$\xymatrix{1\ar[r]^(.35){W}&Cospan(C)}$ such that $C\in\mathbb{C}AT_{push}$,\\ 
 \item $(\mathbb{G}\downarrow j)$ is the comma category of the globular objects 
$\xymatrix{\mathbb{G}\ar[r]^{W}&C}$ in $C$ such that $C\in\mathbb{C}AT_{push}$.\\
\item We have the following isomorphisms of categories
\[\xymatrix{(1\downarrow i\circ Cospan)\ar[r]^(.6){\sim}&(\mathbb{G}\downarrow j)}\qquad \xymatrix{(1\downarrow i\circ Span)\ar[r]^(.55){\sim}&(\mathbb{G}^{op}\downarrow k)}\] 
\end{itemize}
\end{lemma}
 
 \subsection{Digression on $\T$-categories}
\label{T-Cat}

Let us recall the approach to $\omega$-operads by Tom Leinster using 
 $\T$-categories\footnote{For an arbitrary cartesian monad $\mathbb{M}$ on a category with pullbacks the notion of $\mathbb{M}$-category were first suggested by Albert Burroni in 1971; see \cite{burroni-tcat}.}
 (see his book \cite{lein1:oper}). We recall the notions of $\T$-graph and $\T$-category which are also defined in 
 \cite{kach:nscellsfinal,lein1:oper}. Consider the bigategory $\Span(\T)$ as defined in Leinster's book (see \cite{lein1:oper}).
A $\T$-graph $(C,d,c)$ is a diagram of $\OG$ such as
\[\xymatrix{\T(G)&C\ar[l]_(.3){d}\ar[r]^c&G}\]
$\T$-graphs are endomorphisms of $\Span(\T)$ and they form a category $\TG$. 

If we fix $G\in\OG(0)$, the endomorphisms on $G$ in $\Span(\T)$ forms a subcategory of $\TG$ which is denoted $\TG_{G}$. The
category $\TG_{G}$ is monoidal with tensor given by :
\[(C,d,c)\otimes(C',d',c'):=(\T(C)\times_{\T(G)}C',\mu(G)\T(d)\pi_0,c\pi_1),\]
and with unit given by $I(G)=(G,\eta(G),1_G)$. The object $I(G)$ is also an identity morphism of $\Span(\T)$. The $\omega$-graph $G$ is 
called the $\omega$-graph of globular arities, or the $\omega$-graph of arities for short. 

\begin{remark}
\label{notation-couleurs}
   A $p$-cell of $G$ is denoted by $g(p)$ and this notation has the following meaning: The symbol $g$ indicates the "colour", and the symbol $p$ 
   point out that we must see $g(p)$ as a $p$-cell of $G$, because $G$ has to be seen as an $\omega$-graph even though it is just a set.
\end{remark}

A $\T$-graph $(C,d,c)$ equipped with a morphism $I(G)\xrightarrow{p}(C,d,c)$ is called a pointed $\T$-graph. That means that one 
has a $2$-cell $I(G)\xrightarrow{p}(C,d,c)$ of $\Span(\T)$ such that $dp=\eta(G)$ and $cp=1_{G}$. A pointed $\T$-graph is
denoted $(C,d,c;p)$. We define in a natural way the category $\TG_p$ of pointed $\T$-graphs, and also the category 
$\TG_{p,G}$ of $G$-pointed $\T$-graphs: Their morphisms keep pointing in an obvious direction.

A constant $\omega$-graph is an $\omega$-graph $G$ such that $\forall{n,m\in\mathbb{N}}$ we have $G(n)=G(m)$ and such that source and
target maps are identity. We write $\OG_{c}$ for the corresponding category
of constant $\omega$-graphs. We write $\TG_{c}$ for the
subcategory of $\TG$ consisting of $\T$-graphs with underlying
$\omega$-graphs of globular arity which are constant
$\omega$-graphs, and $\TG_{p,c}$ for the subcategory of $\TG_p$ consisting of pointed
$\T$-graphs with underlying $\omega$-graphs of globular arity which are
constant $\omega$-graphs. Also for a given $G$ in $\OG_{c}$, we write $\TG_{p,c,G}$ for the fiber
subcategory in $\TG_{p,c}$.
\begin{definition}
Consider a $\T$-graphs $(C,d,c)$. If $k\geqslant 1$, two $k$-cells $x,y$ of $C$ are parallel if 
 $s^{k}_{k-1}(x)=s^{k}_{k-1}(y)$ and if $t^{k}_{k-1}(x)=t^{k}_{k-1}(y)$. In that case we write $x\|y$.
\end{definition}
 A $\T$-category is a monad in the bicategory $\Span(\T)$ or in an equivalent way a 
monoid of the monoidal category $\TG_{G}$ (for a specific $G$). The category of 
$\T$-categories will denoted $\TC$, and that of $\T$-categories over
the same $\omega$-graph of globular arities $G$ is be denoted
$\TC_{G}$. A $\T$-category $(B,d,c;\gamma,p)\in \TC$ is specifically
given by the morphism of operadic composition
$(B,d,c)\otimes(B,d,c)\xrightarrow{\gamma}(B,d,c)$ and the
operadic unit $I(G)\xrightarrow{p}(B,d,c)$ satisfying axioms of
associativity and unity that we can find in Leinster's book \cite{lein1:oper}.  Note that
$(B,d,c;\gamma,p)$ has $(B,d,c;p)$ as natural underlying pointed
$\T${-}graph. Algebras for a $\T$-category are just algebras for its underlying monad.    
  
\subsection{$\omega$-Operads of Endomorphism and Coendomorphism}
  \label{omega-Operads_of_Endomorphism_and_Coendomorphism}   
  
Let $\mathcal{C}\in\mathbb{G}\mathbb{C}AT$. Recall from \cite{bat:monglob} that the category of 
collections $\omega$-$\mathbb{C}oll(\mathcal{C})$
in $\mathcal{C}$ has as objects globular functors $\xymatrix{\mathbb{T}ree\ar[r]^{A}&\mathcal{C}}$ and as morphisms, globular natural transformations between such globular functors. It is straightforward to see that this defines a strict $2$-functor $\mathbb{C}oll:=Hom_{\mathbb{G}\mathbb{C}AT}(\mathbb{T}ree,-)$

\[\xymatrix{\mathbb{G}\mathbb{C}AT\ar[rr]^{\mathbb{C}oll}&&\mathbb{C}AT}\]

The theorem 6.1 in \cite{bat:monglob} gives criteria for finding many categories of collections with monoidal structure. Colimits commuting with the
monoidal structure of an $AMG$-category is given in definition 5.3 in \cite{bat:monglob}.

\begin{theorem}
\label{theorem61}
If $\mathcal{C}$ is an $AMG$-category such that colimits in $\mathcal{C}$ commute with
its monoidal structure, then $\omega$-$\mathbb{C}oll(\mathcal{C})$ has a natural monoidal structure.
\end{theorem}

For our purpose the main example of such an $AMG$-category 
 as in this theorem is $Span$. The monoidal category $\mathbb{C}oll(Span)$ is equivalent to the monoidal category $\TG_{1}$ of $\T$-graphs over the terminal
 $\omega$-graph $1$ (see \ref{T-Cat} and \cite{lein1:oper}). The category of monoids in $\TG_{1}$ is denoted $\TC_{1}$, and objects of this category are thus
 $\omega$-operads of Batanin in $Span$. So in this article we see the $\omega$-operad $K$\footnote{We prefer to denote it $B^{0}_{C}$ to point out that 
 we consider it as the first step of a sequence of $\omega$-operads : The $2$-coloured $\omega$-operads $B^{n}_{C}$ ($n\geqslant 1$) of the weak higher 
transformations which are object of the category $\TC_{1\sqcup 1}$ of the $\T$-categories over the sum $1\sqcup 1$ of the terminal $\omega$-graph 
 $1$ with itself (see the section \ref{The_Weak_omega_category_of_the_Weak_omega_categories} and the article \cite{kach:nscellsfinal}). 
 The letter "$B$" refer to the name "Batanin", and the subscript $C$ means contractible.} of Batanin as a specific 
$\T$-categories in $\TC_{1}$.
  
Now we are ready to express the main result of this section, which in fact is just a corollary of proposition 7.2 in \cite{bat:monglob}.
\begin{corollary}
\label{le-corollaire}
For each object $(C,W)$ in $(\mathbb{G}\downarrow j)$ we can associate the $\omega$-operad \, $Coend(W)$
of coendomorphisms, given by the following collection 
\[Coend(W):=(HOM(W^{n},W^{t}))_{n\in \mathbb{N}, t\in\mathbb{T}ree}\] 
 Also for each morphism in $(\mathbb{G}\downarrow j)$
\[\xymatrix{(C,W)\ar[r]^{f}&(C',W')}\]
we can associate a morphism of $\omega$-operads 
\[\xymatrix{Coend(W)\ar[rr]^{Coend(f)}&&Coend(W')}\]
Furthermore this construction is functorial; thus it defines a functor
 \[\xymatrix{(\mathbb{G}\downarrow j)\ar[rr]^{Coend}&&\TC_1}\]
Also for each object $(C,W)$ in $(\mathbb{G}^{op}\downarrow k)$ we can associate the $\omega$-operad \, $End(W)$
of endomorphisms, given by the following collection 
\[End(W):=(HOM(W^{t},W^{n}))_{n\in \mathbb{N}, t\in\mathbb{T}ree}.\]   
 Also for each morphism in $(\mathbb{G}^{op}\downarrow k)$
 \[\xymatrix{(C,W)\ar[r]^{f}&(C',W')}\]
we can associate a morphism of $\omega$-operads 
\[\xymatrix{End(W)\ar[rr]^{End(f)}&&End(W')}\]
Furthermore this construction is functorial, thus it defines a functor
 \[\xymatrix{(\mathbb{G}^{op}\downarrow k)\ar[rr]^{End}&&\TC_1}\]
\end{corollary}

 \begin{proposition} 
 \label{la-proposition}      
 If $W\in (\mathbb{G}^{op}\downarrow k)$ then 
         $\xymatrix{End(W)\ar[r]^(.44){\sim}&Coend(W^{op})}$ 
         in $\TC_1$.
 \end{proposition}  
\begin{definition}
If $B\in\TC_1$ then an algebra for $B$ in the sense of Batanin is given by a morphism in $\TC_1$ 
\[\xymatrix{B\ar[rr]^{}&&End(W)}\] 
where $\xymatrix{W: \mathbb{G}^{op}\ar[r]&Set}$ is an object of $\OG$.
\end{definition}
\begin{proposition}[\cite{lein1:oper}]
If $B\in\TC_1$, then an algebra for $B$ in the sense of Batanin, and an algebra for $B$ in the sense of Leinster (see the section \ref{T-Cat}) coincide.
\end{proposition}

\section{Standard actions associated to a coglobular complex in $\TC_c$} 
   \label{Some_standard_diagrams_of_the_omega-transformations}
     
  A $\mathbb{T}$-category over any $\omega$-graph can be seen as a coloured $\omega$-operad (see \cite{kach:nscellsfinal,lein1:oper}), and the category 
 $\TC$ of coloured $\omega$-operads is locally presentable, thus it is a category with pushouts. However it is in the context of
 the locally presentable category $\TC_c$ of $\mathbb{T}$-categories over constant $\omega$-graphs (see paragraph 
 \ref{T-graphs_with_contractible_units} and the article \cite{kach:nscellsfinal}),
that we are going to build \textit{the standard actions associated to a coglobular complex in $\TC_c$}. This concept arises as an application of the previous section
to the category $\TC_c$. Consider the following diagram in $\mathbb{C}AT_{Push}$
  
  \[\xymatrix{\TC_c\ar[rr]^{\mathbb{A}lg(.)}
    &&\mathbb{C}AT^{op}\ar[rr]^{Ob(.)}&&SET^{op} }\]
For each coglobular object $(\mathbb{T}\text{-}\mathbb{C}AT_c,W)$ in $\mathbb{T}$-$\mathbb{C}AT_c$, we have the following diagram in
   $(\mathbb{G}\downarrow j)$
    
  \[\xymatrix{&&\mathbb{G}\ar[dll]_{W}\ar[d]^{A^{op}}\ar[drr]^{A^{op}_{0}}\\
      \TC_c\ar[rr]_{\mathbb{A}lg(.)}&&\mathbb{C}AT^{op}\ar[rr]_{Ob(.)}&&SET^{op} }\]   
If we apply to this diagram the functor $Coend$ of corollary \ref{le-corollaire} and if we use
proposition \ref{la-proposition}, we obtain the following diagram in $\TC_1$
   
   \[\xymatrix{Coend(W)\ar[rr]^{Coend(\mathbb{A}lg(.))}&& 
 Coend(A^{op})\ar[rr]^{Coend(Ob(.))}&&End(A_{0}) }\]
that we call the standard action in $\TC_1$ associated to the coglobular object
 $(\mathbb{T}\text{-}\mathbb{C}AT_c,W)\in (\mathbb{G}\downarrow j)$ in $\mathbb{T}$-$\mathbb{C}AT_c$. 
 
 Now we are ready to explain the philosophy of the standard action associated to a coglobular complex
 in $\TC_1$ : The category $\TC_c$ is locally finitely presentable and the forgetful functor
      \[\xymatrix{\TC_c\ar[rr]^{V}&&\TG_{p,c}}\]
is monadic (see \cite{lein1:oper}), thus according to the proposition 5.5.6 of \cite{francisborceux:handbook2}, $V$ has rank.
 Let us call $M$ its left adjoint. 
 
Now consider a category $P\TC_c$ of $\omega$-operads equipped with a structure that we call "$P$", such that it is locally finitely 
presentable, and equipped with a monadic forgetful functor $\xymatrix{U_{P} : P\TC_c\ar[r]&\TC_c}$. 
Various concrete choices for $P$ will be considered later in this paper.
We denote by $F_{P}$ the left adjoint to $U_{P}$
 \[\xymatrix{P\TC_c\ar[rr]<+4pt>_{\top}^{U_{P}}&&\TC_c\ar[ll]<+6pt>^{F_{P}}\ar[rr]<+4pt>_{\top}^{V}&&\TG_{p,c}\ar[ll]<+6pt>^{M}}\]
Thus we are in a situation where $V\circ U_{P}$ is monadic and the induced monad $\T_P$ on
$\TG_{p,c}$ has rank. Also we get the functor
 \[\xymatrix{P:=F_P\circ M : \TG_{p,c}\ar[rrr]&&&P\TC_c}\]
which assigns the free $PT$-categories on pointed $\T$-graphs.
 
Consider also the following coglobular complex in $\TG_{p,c}$, that we call the 
  coglobular complex for the higher transformations in  $\TG_{p,c}$, because it is built with a combinatoric that we need 
  for higher transformations (see \cite{kach:nscellsfinal}):
  
     \[\xymatrix{C^{0}\ar[rr]<+2pt>^{\delta^{1}_{0}}\ar[rr]<-2pt>_{\kappa^{1}_{0}}
  &&C^{1}\ar[rr]<+2pt>^{\delta^{1}_{2}}\ar[rr]<-2pt>_{\kappa^{2}_{1}}
  &&C^{2}\ar@{.>}[r]<+2pt>^{}\ar@{.>}[r]<-2pt>_{}
  &C^{n-1}\ar[rr]<+2pt>^{\delta^{n}_{n-1}}\ar[rr]<-2pt>_{\kappa^{n}_{n-1}}
  &&C^{n}\ar@{.}[r]<+2pt>\ar@{.}[r]<-2pt>&}\] 
Let us recall the combinatorics involved in this coglobular complex. Pointings $p$ of each collection involved 
in this specific coglobular complex are denoted with the symbol $\lambda$ :
$C^{0}$ is Batanin's system of composition, i.e. there is the
collection $\T(1)\xleftarrow{d^0}C^{0}\xrightarrow{c^0}1$
where $C^0$ precisely contains the symbols of the compositions of the
$\omega$-categories $\mu^{m}_{p}\in{C^{0}(m)}$($0\leq
p<m$), plus the operadic unary symbols $u_{m}\in{C^{0}(m)}$. More
specifically:
\begin{description}
\item $\forall{m}\in\mathbb{N}$, $C^{0}$ contains the $m$-cell $u_{m}$
  such that: $s^{m}_{m-1}(u_{m})=t^{m}_{m-1}(u_{m})=u_{m-1}$ (if $m\geq
  1$); $d^{0}(u_{m})=1(m)(=\eta(1\cup{2})(1(m)))$,
  $c^{0}(u_{m})=1(m)$.
\item $\forall{m}\in\mathbb{N}-\{0,1\}$, $\forall{p}\in\mathbb{N}$,
  such that $m>p$, $C^{0}$ contains the $m$-cell $\mu^{m}_{p}$ such
  that: If $p=m-1$,
  $s^{m}_{m-1}(\mu^{m}_{m-1})=t^{m}_{m-1}(\mu^{m}_{m-1})=u_{m-1}$.  If
  $0\leq p<m-1$,
  $s^{m}_{m-1}(\mu^{m}_{p})=t^{m}_{m-1}(\mu^{m}_{p})=\mu^{m-1}_{p}$.
  Also $d^{0}(\mu^{m}_{p})=1(m)\star^{m}_{p}1(m)$, and inevitably
  $c^{0}(\mu^{m}_{p})=1(m)$.
\item Furthemore $C^{0}$ contains the $1$-cell $\mu^{1}_{0}$ such that
   $s^{1}_{0}(\mu^{1}_{0})=t^{1}_{0}(\mu^{1}_{0})=u_{0}$,
    $d^{0}(\mu^{1}_{0})=1(1)\star^{1}_{0}1(1)$, also inevitably
  $c^{0}(\mu^{1}_{0})=1(1)$.
\end{description}
The system of composition $C^0$ has got a well-known pointing
$\lambda^{0}$ which is defined by: $\forall{m}\in\mathbb{N}$,
$\lambda^{0}(1(m))=u_{m}$.

Firstly we will define a collection $(C,d,c)$ which will be useful to
build the collections of $n$-transformations ($n\in {\mathbb{N}^*}$). $C$ contains two copies of
the symbols of $C^0$, each having a distinct colour: The
symbols formed with the letters $\mu$ and $u$ are those of colour $1$, and
those formed with the letters $\nu$ and $v$ are those of colour
$2$. Let us be more precise:
\begin{description}                     
\item $\forall{m}\in\mathbb{N}$, $C$ contains the $m$-cell $u_{m}$
  such that: $s^{m}_{m-1}(u_{m})=t^{m}_{m-1}(u_{m})=u_{m-1}$ (if $m\geq
  1$) and $d(u_{m})=1(m)$, $c(u_{m})=1(m)$.
\item $\forall{m}\in\mathbb{N}-\{0,1\}$, $\forall{p}\in\mathbb{N}$,
  such that $m>p$, $C$ contains the $m$-cell $\mu^{m}_{p}$ such that: If
  $p=m-1$,
  $s^{m}_{m-1}(\mu^{m}_{m-1})=t^{m}_{m-1}(\mu^{m}_{m-1})=u_{m-1}$.  If
  $0\leq p<m-1$,
  $s^{m}_{m-1}(\mu^{m}_{p})=t^{m}_{m-1}(\mu^{m}_{p})=\mu^{m-1}_{p}$.
  Also $d(\mu^{m}_{p})=1(m)\star^{m}_{p}1(m)$, $c(\mu^{m}_{p})=1(m)$.
\item Furthemore $C$ contains the $1$-cell $\mu^{1}_{0}$ such that
  $s^{1}_{0}(\mu^{1}_{0})=t^{1}_{0}(\mu^{1}_{0})=u_{0}$ and
  $d(\mu^{1}_{0})=1(1)\star^{1}_{0}1(1)$, $c(\mu^{1}_{0})=1(1)$.
\item Besides, $\forall{m}\in\mathbb{N}$, $C$ contains the $m$-cell
  $v_{m}$ such that: $s^{m}_{m-1}(v_{m})=t^{m}_{m-1}(v_{m})=v_{m-1}$
  (if $m\geq 1$) and $d(v_{m})=2(m)$, $c(v_{m})=2(m)$.
\item $\forall{m}\in\mathbb{N}-\{0,1\}$, $\forall{p}\in\mathbb{N}$,
  such that $m>p$, $C$ contains the $m$-cell $\nu^{m}_{p}$ such that: If
  $p=m-1$,
  $s^{m}_{m-1}(\nu^{m}_{m-1})=t^{m}_{m-1}(\nu^{m}_{m-1})=v_{m-1}$.  If
  $0\leq p<m-1$,
  $s^{m}_{m-1}(\nu^{m}_{p})=t^{m}_{m-1}(\nu^{m}_{p})=\nu^{m-1}_{p}$.
  Also $d(\nu^{m}_{p})=2(m)\star^{m}_{p}2(m)$, $c(\nu^{m}_{p})=2(m)$.
\item Furthemore $C$ contains the $1$-cell $\nu^{1}_{0}$ such that
  $s^{1}_{0}(\nu^{1}_{0})=t^{1}_{0}(\nu^{1}_{0})=v_{0}$ and
  $d(\nu^{1}_{0})=2(1)\star^{1}_{0}2(1)$, $c(\nu^{1}_{0})=2(1)$.
\end{description}

$C^1$ is the system of operations of $\omega$-functors.  It
is built on the basis of $C$ adding to it a single symbol of
functor (for each cell level):$\forall{m}\in\mathbb{N}$ the $F^{m}$ $m$-cell
is added, which is such that: If $m\geq
1$, $s^{m}_{m-1}(F^{m})=t^{m}_{m-1}(F^{m})=F^{m-1}$.  Also
$d^{1}(F^{m})=1(m)$ and $c^{1}(F^{m})=2(m)$.

$C^2$ is the system of operations of the natural
$\omega$-transformations. $C^2$ is built on $C$, adding to it two
symbols of functor (for each cell level) and a symbol of natural
transformation. More precisely
\begin{description}
\item $\forall{m}\in\mathbb{N}$ we add the $m$-cell $F^{m}$
  such that: If $m\geq 1$,
  $s^{m}_{m-1}(F^{m})=t^{m}_{m-1}(F^{m})=F^{m-1}$.  Also
  $d^{2}(F^{m})=1(m)$ and $c^{2}(F^{m})=2(m)$.
\item Then $\forall{m}\in\mathbb{N}$ we add the $m$-cell $H^{m}$ 
  such that: If $m\geq
  1$, $s^{m}_{m-1}(H^{m})=t^{m}_{m-1}(H^{m})=H^{m-1}$.  Also
  $d^{2}(H^{m})=1(m)$ and $c^{2}(H^{m})=2(m)$.
\item And finally we add $1$-cell $\tau$ such that:
  $s^{1}_{0}(\tau)=F^{0}$ and $t^{1}_{0}(\tau)=H^{0}$.  Also
  $d^{2}(\tau)=1_{1(0)}$ and $c^{2}(\tau)=2(1)$.
\end{description}
We can point out that the $2$-coloured collections $C^{i}$ ($i=1,2$)
are naturally equipped with a pointing $\lambda^{i}$ defined by
$\lambda^{i}(1(m))=u_{m}\text{ and }\lambda^{i}(2(m))=v_{m}$.

In order to define the general theory of the $n$-transformations ($n\in {\mathbb{N}^*}$), it is
necessary to define the systems of operations $C^{n}$ for the superior
$n$-transformations ($n\geq 3$). This paragraph can be left out in
the first reading. Each collection $C^{n}$ is built on $C$, adding to
it the required cells. They contain four large groups of cells: The
symbols of source and target $\omega$-categories, the
symbols of operadic units (obtained on the basis of $C$), the symbols
of the $\omega$-functors (sources and targets), and the symbols of the $n$-transformations
(natural $\omega$-transformations, $\omega$-modification, etc). More precisely, on the
basis of $C$:
\begin{description}
\item[Symbols of the $\omega$-Functors] $\forall{m}\in\mathbb{N}$, $C^{n}$ contains
  the $m$-cells $\alpha^{m}_{0}$ and $\beta^{m}_{0}$ such as: If
  $m\geq 1$,
  $s^{m}_{m-1}(\alpha^{m}_{0})=t^{m}_{m-1}(\alpha^{m}_{0})=\alpha^{m-1}_{0}$
  and
  $s^{m}_{m-1}(\beta^{m}_{0})=t^{m}_{m-1}(\beta^{m}_{0})=\beta^{m-1}_{0}$.
  Furthermore $d^{n}(\alpha^{m}_{0})=d^{n}(\beta^{m}_{0})=1(m)$ and
  $c^{n}(\alpha^{m}_{0})=c^{n}(\beta^{m}_{0})=2(m)$.
\item[Symbols of the Higher $n$-Transformations]
  $\forall{p}$, with $1\leq p$ $\leq n-1$, $C^{n}$ contains the $p$-cells
  $\alpha_{p}$ and $\beta_{p}$ which are such as: $\forall{p}$, with
  $2\leq p\leq n-1$,
  $s^{p}_{p-1}(\alpha_{p})=s^{p}_{p-1}(\beta_{p})=\alpha_{p-1}$ and
  $t^{p}_{p-1}(\alpha_{p})=t^{p}_{p-1}(\beta_{p})=\beta_{p-1}$. If
  $p=1$, $s^{1}_{0}(\alpha_{1})=s^{1}_{0}(\beta_{1})=\alpha^{0}_{0}$
  and $t^{1}_{0}(\alpha_{1})=t^{1}_{0}(\beta_{1})=\beta^{0}_{0}$.
  What's more, $\forall{p}$, with $1\leq p\leq n-1$,
  $d^{n}(\alpha_{p})=d^{n}(\beta_{p})=1^{0}_{p}(1(0))$ and
  $c^{n}(\alpha_{p})=c^{n}(\beta_{p})=2(p)$.  Finally $C^{n}$ contains
  the $n$-cell $\xi_{n}$ such that $s^{n}_{n-1}(\xi_{n})=\alpha_{n-1}$,
  $b^{n}_{n-1}(\xi_{n})=\beta_{n-1}$ and
  $d^{n}(\xi_{n})=1^{0}_{n}(1(0))$ and $c^{n}(\xi_{n})=2(n)$.
\end{description}
We can see that $\forall{n}\in{\mathbb{N}^{\ast}}$, the $2$-colored
collection $C^{n}$ is naturally equipped with the pointing
 $1\cup 2\xrightarrow{\lambda^{n}}(C^{n},d,c)$ defined
as:
 \[\forall{m}\in{\mathbb{N}},\lambda^{n}(1(m))=u_{m} \text{ and }
\lambda^{n}(2(m))=v_{m}.\]
 The set
$\{C^{n}/n\in\mathbb{N}\}$ has a canonical structure of
coglobular complex.  This coglobular complex is generated by diagrams
\[\xymatrix{C^{n}\ar@<+2pt>[rr]^(.5){\delta^{n}_{n+1}}\ar@<-2pt>[rr]_(.5){\kappa^{n}_{n+1}}&&C^{n+1}}\]
of pointed $2$-coloured collections.  For $n\geq 2$, these diagrams
are defined as follows: First the $(n+1)$-colored collection contains
the same symbols of operations as $C^{n}$ for the $j$-cells, $0\leq j\leq
n-1$ or $n+2\leq j<\omega$. For the $n$-cells and the $(n+1)$-cells the
symbols of operations will change: $C^{n}$ contains the $n$-cell
$\xi_{n}$ whereas $C^{n+1}$ contains the $n$-cells $\alpha_{n}$ and
$\beta_{n}$, in addition contains the $(n+1)$-cell $\xi_{n+1}$.  If
one denotes by $C^{n}-{\xi_{n}}$ the $n$-coloured collection obtained on
the basis of $C^{n}$ by taking from it the $n$-cell $\xi_{n}$, then
$\delta^{n}_{n+1}$ is defined as follows:
$\delta^{n}_{n+1}\vert_{C^{n}-{\xi_{n}}}$ (i.e the restriction of
$\delta^{n}_{n+1}$ to $C^{n}-{\xi_{n}}$) is the canonical injection
$C^{n}-{\xi_{n}}\hookrightarrow C^{n+1}$ and
$\delta^{n}_{n+1}(\xi_{n})=\alpha_{n}$.  In a similar way
$\kappa^{n}_{n+1}$ is defined as follows:
$\kappa^{n}_{n+1}\vert_{C^{n}-{\xi_{n}}}=\delta^{n}_{n+1}\vert_{C^{n}-{\xi_{n}}}$
and $\kappa^{n}_{n+1}(\xi_{n})=\beta_{n}$. We can notice that
$\delta^{n}_{n+1}$ and $\kappa^{n}_{n+1}$ keeps pointing, i.e we have
for all $n\geq 1$ the equalities
$\delta^{n}_{n+1}\lambda^{n}=\lambda^{n+1}$ and
$\kappa^{n}_{n+1}\lambda^{n}=\lambda^{n+1}$.

The morphisms of $2$-colored pointing collections of the diagram
\[\xymatrix{C^{0}\ar@<+2pt>[r]^{\delta^{0}_{1}}\ar@<-2pt>[r]_{\kappa^{0}_{1}}&C^{1}
  \ar@<+2pt>[r]^{\delta^{1}_{2}}\ar@<-2pt>[r]_{\kappa^{1}_{2}}&C^{2}
  \ar@<+2pt>[r]^{\delta^{2}_{3}}\ar@<-2pt>[r]_{\kappa^{2}_{3}}&C^{3}}\]
have a similar definition:

We have for all integers $0\leq p<n$ and for all $\forall{m}\in\mathbb{N}$:
\begin{description}
\item $\delta^{0}_{1}(\mu^{n}_{p})=\mu^{n}_{p}$;
  $\delta^{0}_{1}(u_{m})=u_{m}$;
  $\kappa^{0}_{1}(\mu^{n}_{p})=\nu^{n}_{p}$;
  $\kappa^{0}_{1}(u_{m})=v_{m}$.
\item Also: $\delta^{1}_{2}(\mu^{n}_{p})=\mu^{n}_{p}$;
  $\delta^{1}_{2}(u_{m})=u_{m}$;
  $\delta^{1}_{2}(\nu^{n}_{p})=\nu^{n}_{p}$;
  $\delta^{1}_{2}(v_{m})=v_{m}$; $\delta^{1}_{2}(F^{m})=F^{m}$. And
  $\kappa^{1}_{2}(\mu^{n}_{p})=\mu^{n}_{p}$;
  $\kappa^{1}_{2}(u_{m})=u_{m}$;
  $\kappa^{1}_{2}(\nu^{n}_{p})=\nu^{n}_{p}$;
  $\kappa^{1}_{2}(v_{m})=v_{m}$; $\kappa^{1}_{2}(F^{m})=H^{m}$.
\item Finally: $\delta^{2}_{3}(\mu^{n}_{p})=\mu^{n}_{p}$;
  $\delta^{2}_{3}(u_{m})=u_{m}$;
  $\delta^{2}_{3}(\nu^{n}_{p})=\nu^{n}_{p}$;
  $\delta^{2}_{3}(v_{m})=v_{m}$;
  $\delta^{2}_{3}(F^{m})=\alpha^{m}_{0}$;
  $\delta^{2}_{3}(H^{m})=\beta^{m}_{0}$;
  $\delta^{2}_{3}(\tau)=\alpha_{1}$. And
  $\kappa^{2}_{3}(\mu^{n}_{p})=\mu^{n}_{p}$;
  $\kappa^{2}_{3}(u_{m})=u_{m}$;
  $\kappa^{2}_{3}(\nu^{n}_{p})=\nu^{n}_{p}$;
  $\kappa^{2}_{3}(v_{m})=v_{m}$; 
  $\kappa^{2}_{3}(F^{m})=\alpha^{m}_{0}$;
  $\kappa^{2}_{3}(H^{m})=\beta^{m}_{0}$;
  $\kappa^{2}_{3}(\tau)=\beta_{1}$.
\end{description}
The pointed $2$-coloured collections $C^{n}$ $(n\in\mathbb{N}^*)$ are
the systems of operations of the $n$-transformations.  
  
If we apply the functor $P$ to this coglobular complex we obtain a  coglobular complex in  $P\TC_c$
    
    \[\xymatrix{B^{0}_{P}\ar[rr]<+2pt>^{\delta^{1}_{0}}\ar[rr]<-2pt>_{\kappa^{1}_{0}}
  &&B^{1}_{P}\ar[rr]<+2pt>^{\delta^{1}_{2}}\ar[rr]<-2pt>_{\kappa^{2}_{1}}
  &&B^{2}_{P}\ar@{.>}[r]<+2pt>^{}\ar@{.>}[r]<-2pt>_{}
  &B^{n-1}_{P}\ar[rr]<+2pt>^{\delta^{n}_{n-1}}\ar[rr]<-2pt>_{\kappa^{n}_{n-1}}
  &&B^{n}_{P}\ar@{.}[r]<+2pt>\ar@{.}[r]<-2pt>&}\]
which is also, when we forget its structure "$P$", a coglobular object $W=B^{\bullet}_{P}$ of $\TC_c$,
  and thus we obtain its resulting standard action
  
   \[\xymatrix{Coend(B^{\bullet}_{P})\ar[rr]^{Coend(\mathbb{A}lg(.))}&& 
 Coend(A_{P}^{op})\ar[rr]^{Coend(Ob(.))}&&End(A_{0,P}) }\]
where in particular $Coend(B^{\bullet}_{P})$ is the monochromatic $\omega$-operad of coendomorphism associated to this
   coglobular complex. This kind of standard actions is called a \textit{standard action for higher transformations} because 
   it is built with the coglobular complex of the higher transformations $C^{\bullet}$ in  $\TG_{p,c}$.  
      
   The main problem of our philosophy, is to build a morphism of $\omega$-operads between the monochromatic $\omega$-operad $B^{0}_{_{P}}$ (the "$0$-step"
   of the coglobular object $B^{\bullet}_{P}$) and the monochromatic $\omega$-operad $Coend(B^{\bullet}_{P})$ (built with the whole coglobular object 
   $B^{\bullet}_{P}$). 
If such a morphism exists then we have a morphism of operads 
 \[\xymatrix{B^{0}_{P}\ar[rrr]&&&End(A_{0,P})}\]
which shows that $B^{0}_{P}$-algebras and all its higher transformations form a $B^{0}_{P}$-algebra. 
In this case we say that $B^{0}_{P}$ has the \textit{fractal property}.

   \section{Contractibility Hypotheses}
  \label{Contractibility_Hypothesis_and_the_Weak_omega_category_of_the_Weak_omega_categories}
  
In this paragraph we will consider two cases : When $P$ is \textit{strict with contractible units} (indicated with the letters $S_{u}$), and when $P$ is 
   \textit{contractible} (indicated with the letter $C$). We will state the hypotheses that the $\omega$-operad $B^{0}_{S_{u}}$ of strict $\omega$-categories, and
   the $\omega$-operad $B^{0}_{C}$ of the weak $\omega$-categories, have the fractal property. 

In the section \ref{Examples_of_Standard_Actions} we give two examples of other higher structure such that it is possible to prove that their 
associated $\omega$-operads $B^{0}_{P}$ have the fractal property. 
  
  \subsection{The functor of the contractible units}
    \label{T-graphs_with_contractible_units}  
  In this paragraph we are going to build 
 the \textit{functor of the contractible units} for $\T$-categories (see \ref{T-Cat}). But first we must define the 
 \textit{pointed $\T$-graphs with contractible units} which are for the pointed $\T$-graphs what the reflexive $\omega$-graphs are for 
 the $\omega$-graphs. In order to define it we are going first to define an intermediary structure on $\T$-graphs. Consider a 
    $\T$-graph $(C,d,c)$, and for each $n\in\mathbb{N}$ we write $C(n)$ for the set of $n$-cells of the $\T$-graph $(C,d,c)$. 
  Note that $G$ is also equipped with a trivial
    reflexivity structure $(G,(1^{p}_{n})_{0\leqslant p<n}))$ 
    where the operations $1^{p}_{n}$ are defined by $1^{p}_{n}(g(p))=g(n)$, and which force
    $c$ to be a morphism of reflexive $\omega$-graphs as well.  
\begin{definition}    
    We say that the $\T$-graph $(C,d,c)$ is equipped with a reflexive structure, if its underlying $\omega$-graph $C$ is equipped 
    with a reflexive structure in the usual sense, such that $d$ is a morphism of reflexive $\omega$-graphs. 
\end{definition}    
     We denote $(C,d,c; (1^{p}_{n})_{0\leqslant p<n})$ a reflexive $\T$-graph where 
    the operations $1^{p}_{n}$ are those of $C$. A morphism between two reflexive $\T$-graphs are just morphism of $\T$-graphs which preserve
    reflexivity, and the category of reflexive $\T$-graphs over constant $\omega$-graphs is denoted by 
    $\TGr_c$.
    
  A pointed $\T$-graph $(C,d,c;p)$ over a constant $\omega$-graph $G$ has contractible units if it is equipped with a monomorphism
    $\xymatrix{\R(G)\ar@{>->}[r]^{i}&C}$ such that $p$ factorises as follow
    
    \[\xymatrix{&&C\ar[lld]_{d}\ar[rrd]^{c}&&\\
    \T(G)&&\R(G)\ar@{>->}[u]^{i}\ar[ll]_{d}\ar[rr]^{c}&&G\\
    &&G\ar[u]^{\eta(G)}\ar[llu]^{\eta(G)}\ar[rru]_{id}}\]      
     such that $p=i\eta(G)$, and such that the induced $\T$-graph: 
     \[\xymatrix{\T(G)&\R(G)\ar[l]_{d}\ar[r]^{c}&G}\]
 is reflexive, i.e the restriction of $d$ on $\R(G)$ is a morphism of reflexive $\omega$-graphs. 
 We let $(C,d,c; p,i,(1^{p}_{n})_{0\leqslant p<n})$ be a pointed $\T$-graph with contractible units. 
 Morphisms of pointed $\T$-graphs with contractible units
   
   \[\xymatrix{(C,d,c; p,i,(1^{p}_{n})_{0\leqslant p<n})\ar[rr]^{(f,h)}&&(C',d',c'; p',i',(1^{p}_{n})_{0\leqslant p<n})}\]
is given by morphisms of pointed $\T$-graphs (see \cite{lein1:oper})
   \[\xymatrix{(C,d,c; p)\ar[rr]^{(f,h)}&&(C',d',c'; p')}\]
such that $fi=i'\R(h)$. 
 The category of pointed $\T$-graphs with contractible units is denoted by $Id_{u}\TG_{p,c}$, and
 a $T$-category has contractible units if its underlying pointed $T$-graph lies in $Id_{u}\TG_{p,c}$.
 Morphisms between two $T$-categories equipped with contractible units are just morphisms
 of $T$-categories which preserve contractible units. Let us write $Id_{u}\TC_c$ this category. 
  It is a locally presentable category, and also we can prove that the forgetful functor is monadic.
    \[\xymatrix{U\TC_c\ar[r]^{U_{Id_{u}}}&\TC_c}\]
    Let us note $F_{Id_{u}}\dashv U_{Id_{u}}$ this adjunction. 
In particular we get the functor of the free $\T$-categories with contractible units $\xymatrix{Id_{u}:\TG_{p,c}\ar[r]^{}&U\TC_c}$
which is a left adjoint 
 and which monad $\T_{Id_{u}}$ has rank. 
  
\subsection{The functors of strictification and the functors of contractibility}
   \label{Weak_contractibility_and_strict_contractibility} 
   
  In \cite{kach:nscellsfinal} we defined a coglobular complex of $\omega$-operads
 
  \[\xymatrix{B^{0}\ar[rr]<+2pt>^{\delta^{1}_{0}}\ar[rr]<-2pt>_{\kappa^{1}_{0}}
  &&B^{1}\ar[rr]<+2pt>^{\delta^{2}_{1}}\ar[rr]<-2pt>_{\kappa^{2}_{1}}
  &&B^{2}\ar@{.>}[r]<+2pt>^{}\ar@{.>}[r]<-2pt>_{}
  &B^{n-1}\ar[rr]<+2pt>^{\delta^{n}_{n-1}}\ar[rr]<-2pt>_{\kappa^{n}_{n-1}}
  &&B^{n}\ar@{.}[r]<+2pt>\ar@{.}[r]<-2pt>&}\]
such that algebras for $B^{0}$ are the weak $\omega$-categories, algebras for $B^{1}$ are the weak
  $\omega$-functors, algebras for $B^{2}$ are the weak $\omega$-natural transformations, etc. However
  Andr{\'e} Joyal has pointed out to us that there are too many coherence cells for each $B^{n}$ when
  $n\geqslant 2$, and gave us a simple example of a natural transformation which cannot be
  an algebra for the $2$-coloured $\omega$-operad $B^{2}$. 
   
  We are going to propose a notion of contractibilty, slightly different from those used in \cite{bat:monglob,kach:nscellsfinal}. 
  This new approach excludes the counterexample of Andr{\'e} Joyal, but also shows that with a \textit{strict version}
  of this corrected version, algebras for the specific \textit{strict $\omega$-operads $B^{n}_{S_{u}}$ with contractible units ($n\geqslant 1$)} 
  of the higher transformations that we propose (see section \ref{The_Strict_omega_category_of_the_Strict_omega_categories}), 
  follow the axioms of the strict higher transformations. 
  
  Let us speak now about our own intuition about this new notion of contractibility that we propose: 
  Roughly speaking, it deals with notions of \textit{root cells plus the loop condition}. It is first based on a basic observation about
  the contractibility of the Batanin's operad $B^{0}_{C}$ : The pairs of cells in $B^{0}_{C}$, say 
  $(x,y)\in B^{0}_{C}(n)\times B^{0}_{C}(n)$, which are parallels and have same arity must be connected by a coherence cells, but notice that they also have the 
  following \textit{extra property}: $s^{n}_{0}(x)=s^{n}_{0}(y)=t^{n}_{0}(x)=t^{n}_{0}(y)$.
  By \textit{extra} we mean that this is extra over the usual definition of contractibility in the sense of Batanin, where normalised $\omega$-operads 
  (which by definition follow this extra condition; see \cite{bat:monglob}), are not explicitly considered in his approach.  
  Let us call this property \textit{the loop property}.
 \begin{definition}
  For any $\mathbb{T}$-graph $(C,d,c)$ over a constant $\omega$-graph $G$,
   a pair of cells $(x,y)$ of $C(n)$ has the \textit{the loop property} if : $s^{n}_{0}(x)=s^{n}_{0}(y)=t^{n}_{0}(x)=t^{n}_{0}(y)$.
 \end{definition}  
  Secondly it concerns the correct $\omega$-operads
   $B^{i}_{C} (i\in \mathbb{N}^{*})$ of the weak higher transformations that we are looking for: In our approach
   these $\omega$-operads should be freely generated by the $\mathbb{T}$-graphs $C^{i} (i\in \mathbb{N}^{*})$ of the weak
   higher transformations (see section \ref{Some_standard_diagrams_of_the_omega-transformations}). We observe the important fact that all
   the symbols specific for the higher transformations in it, have arities which are the reflexivity of $1(0)$, where $1(0)$ denotes the unique $0$-cell of 
  colour $1$ of $\mathbb{T}(1+2)$, i.e $\forall n\geqslant 1$, those cells $x\in C(n)$ specific for the higher transformations are such that
  $d(x)=1^{0}_{n}(1(0))$. Let us call these kind of cells \textit{the root cells}.
  \begin{definition}
  For any $\mathbb{T}$-graph $(C,d,c)$ over a constant $\omega$-graph $G$,
   we call \textit{the root cells} of $(C,d,c)$, those cells whose arities are the reflexivity of a $0$-cell $g(0)$ of $G$, where here "$g$" 
   indicates the colour (see section \ref{notation-couleurs}), or in other words, those cells $x\in C(n)$ ($n\geqslant 1$) such that
   $d(x)=1^{0}_{n}(g(0))$.
 \end{definition}
These notions of \textit{root cells} and \textit{loop condition} are the keys for
  our new approach to contractibility.   
 These observations motivate us to put
  the following definition of what should be a contractible $\T$-graphs $(C,d,c)$.
  For each integers $k\geqslant 1$ let us note
  $\tilde{C}(k)=\{(x,y)\in C(k)\times C(k): x\|y$ 
    and $d(x)=d(y)$, and if also $(x,y)$ is a pair of root cells then they also need to verify the \textit{loop
    property}: 
    $s^{k}_{0}(x)=t^{k}_{0}(y)\}$. Also we denote $\tilde{C}(0)=\{(x,x)\in C(0)\times C(0)\}$.
\begin{definition}
 A contraction on the $\mathbb{T}$-graph $(C,d,c)$, is the datum, for all $k\in\mathbb{N}$, of a map
$\tilde{C}(k)\xrightarrow{[,]_{k}}C(k+1)$ such that
\begin{itemize}
\item$s([\alpha,\beta]_{k})=\alpha, t([\alpha,\beta]_{k})=\beta$,
\item$d([\alpha,\beta]_{k})=1_{d(\alpha)=d(\beta)}$.
\end{itemize}
\end{definition}
A $\mathbb{T}$-graph which is equipped with a
contraction will be called contractible and we use the notation
$(C,d,c;([,]_{k})_{k\in\mathbb{N}})$ for a contractible $\mathbb{T}$-graph.
Nothing prevents a contractible $\mathbb{T}$-graph from being equipped with
several contractions. So here $C\TG_c$ is the category of the contractible
$\mathbb{T}$-graphs equipped with a specific contraction, and morphisms of this
category preserves the contractions. One can also refer to the
category $C\TG_{c,G}$, where here contractible $\mathbb{T}$-graphs are only taken over a specific
constant $\infty$-graph $G$. A pointed contractible $\mathbb{T}$-graphs 
 (see section \ref{T-Cat}) is denoted $(C,d,c;p,([,]_{k})_{k\in\mathbb{N}})$, 
and morphisms between two pointed contractible $\mathbb{T}$-graphs preserve contractibilities and pointings. The category 
of pointed contractible $\mathbb{T}$-graphs is denoted $C\TG_{p,c}$. Objects of the category $C\TG_{p,c}$ 
 are examples of $\mathbb{T}$-graphs equipped with contractible units
 (see section \ref{T-graphs_with_contractible_units}), and $C\mathbb{T}$-$\G_{p,c}$ is a subcategory of $Id_{u}\mathbb{T}$-$\G_{p,c}$.
 A $\T$-category is contractible if its underlying pointed $\T$-graph lies in $C\mathbb{T}$-$\G_{p,c}$. Morphisms between
 two contractible $\T$-categories are morphisms of $\T$-categories which preserve contractibilities. Let us write
 $C\TC_c$ for the category of contractible $\T$-categories.
\begin{remark}
 It is evident that the $\omega$-operad $B^{0}_{C}$ of Michael Batanin is still initial in the category of  
 contractible $\omega$-operads equipped with a composition system, where our new approach of  contractibility is considered.  
\end{remark}
$C\TC_c$ is a locally presentable category, and also we can prove that the forgetful functor is monadic.
  \[\xymatrix{C\TC_c\ar[r]^{U_{C}}&\TC_c}\]
   Let us write $F_{C}\dashv U_{C}$ this adjunction. 
In particular we get the functor of the free contractible $\T$-categories 
$\xymatrix{C : \TG_{p,c}\ar[r]^{}&C\TC_c}$, which is a left adjoint and
whose monad $\T_C$ has rank.

 Now let us explain the notion of \textit{strict contractibility} on a $\mathbb{T}$-graph $(C,d,c)$. We say
that $(C,d,c)$ is strictly contractible if for each integer $k\in\mathbb{N}$ and
each $(x,y)\in\tilde{C}(k)$ we have $x=y$. This notion of strict contractibility plus the notion of 
$\mathbb{T}$-graphs equipped with contractible units (see section \ref{T-graphs_with_contractible_units}), is going to be used to define a notion of 
$\omega$-operads for strict higher transformations (see section \ref{The_Strict_omega_category_of_the_Strict_omega_categories}). It is evident that a morphism in $\mathbb{T}$-$\G_c$ between two strictly contractible $\mathbb{T}$-graphs 
preserves strict contractibility. So we write $S\mathbb{T}$-$\G_c$ for the category of strictly contractible
$\mathbb{T}$-graphs as a full subcategory of $\mathbb{T}$-$\G_c$. The definition of pointed strictly contractible $\mathbb{T}$-graphs 
is obvious : It is just the strictly contractible $\mathbb{T}$-graphs equipped with pointings. Morphisms between 
two pointed strictly contractible $\mathbb{T}$-graphs are morphisms of $\mathbb{T}$-$\G_c$ which preserve pointings. We
write $S\mathbb{T}$-$\G_{p,c}$ the category of pointed strictly contractible $\mathbb{T}$-graphs. 
A $\T$-category is strictly contractible if its underlying pointed $\T$-graph is strictly contractible. Morphisms
between two strictly contractible $\T$-categories are morphisms of $\T$-categories which preserve the
strict contractibilities. Let us write $S\TC_c$ for the category of strictly contractible $\T$-categories.
It is a locally presentable category, and also we can prove that the forgetful functor is monadic.
  \[\xymatrix{S\TC_c\ar[r]^{U_{S}}&\TC_c}\]
    Denote this adjunction by $F_S  \dashv 
    U_S$. 
In particular we get the functor $S$ for the free strictly contractible $\T$-categories 
\[\xymatrix{S : \TG_{p,c}\ar[r]^{}&S\TC_c}\]
which is a left adjoint and whose monad $\T_S$ has rank.

 Two kind of monads on $\TG_{p,c}$ are relevant for us : The monad $\T_C$ and the coproduct of monads 
 $\T_{S_{u}}:=\T_{Id_{u}}\coprod\T_{S}$ which also has rank, and whose category of algebras is denoted $S_{u}\TC_c$. 
In particular the monad $\T_{S_{u}}$ gives the functor $S_{u}$ of free strict $\T$-categories equipped with 
contractible units 
\[\xymatrix{S_u:\TG_{p,c}\ar[r]^{}&S_{u}\TC_c}\] 
The functor $S_u$ is used to build the coglobular complex of the $\omega$-operads of the higher transformations for strict $\omega$-categories 
 (see \ref{The_Strict_omega_category_of_the_Strict_omega_categories}). The functor $C$ is used to build the coglobular complex of 
 the $\omega$-operads of the higher transformations for weak $\omega$-categories (see \ref{The_Weak_omega_category_of_the_Weak_omega_categories}).

\subsection{Contractibility Hypotheses}
 \label{Contractibility_Hypothesis}
 
 By using functors of the previous section 
 \[\xymatrix{S_u : \TG_{p,c}\ar[r]^{}&S_{u}\TC_c}\qquad\xymatrix{C : \TG_{p,c}\ar[r]^{}&C\TC_c}\]  
with the coglobular complexes of the higher transformations $C^{\bullet}$ in $\TG_{p,c}$ we get two coglobular complex in $\TC_c$ :
 
 \[\xymatrix{B_{S_{u}}^{0}\ar[rr]<+2pt>^{\delta^{1}_{0}}\ar[rr]<-2pt>_{\kappa^{1}_{0}}
  &&B_{S_{u}}^{1}\ar[rr]<+2pt>^{\delta^{1}_{2}}\ar[rr]<-2pt>_{\kappa^{2}_{1}}
  &&B_{S_{u}}^{2}\ar@{.>}[r]<+2pt>^{}\ar@{.>}[r]<-2pt>_{}
  &B_{S_{u}}^{n-1}\ar[rr]<+2pt>^{\delta^{n}_{n-1}}\ar[rr]<-2pt>_{\kappa^{n}_{n-1}}
  &&B_{S_{u}}^{n}\ar@{.}[r]<+2pt>\ar@{.}[r]<-2pt>&}\]     
 \[\xymatrix{B_{C}^{0}\ar[rr]<+2pt>^{\delta^{1}_{0}}\ar[rr]<-2pt>_{\kappa^{1}_{0}}
  &&B_{C}^{1}\ar[rr]<+2pt>^{\delta^{1}_{2}}\ar[rr]<-2pt>_{\kappa^{2}_{1}}
  &&B_{C}^{2}\ar@{.>}[r]<+2pt>^{}\ar@{.>}[r]<-2pt>_{}
  &B_{C}^{n-1}\ar[rr]<+2pt>^{\delta^{n}_{n-1}}\ar[rr]<-2pt>_{\kappa^{n}_{n-1}}
  &&B_{C}^{n}\ar@{.}[r]<+2pt>\ar@{.}[r]<-2pt>&}\]
 where $B_{S_{u}}^{0}$ is the $\omega$-operad for strict $\omega$-categories and $B_{C}^{0}$ is the $\omega$-operad of Batanin for
  weak $\omega$-categories.  
 In particular the $\omega$-operads $B^{0}_P$, where $P$ is either $S_{u}$ 
 for \textit{strict with contractible units} (see section \ref{The_Strict_omega_category_of_the_Strict_omega_categories}), 
 or $C$ for \textit{contractible} (see section \ref{The_Weak_omega_category_of_the_Weak_omega_categories}), 
 have interesting common characteristics: Each  $B^{0}_P$ has a universal property in 
 $P\TC_c$, which is to be initial among $\omega$-operads (monochromatic or even coloured) equipped with a composition system (explicitly given
 by the collection $C^{0}$), and satisfying the property $P$. Thus if we show, for a fixed property $P$, that $coend(B^{\bullet}_P)$ is at the same time
 equipped with a composition system and verifies the property $P$, then $B^{0}_P$ must have the fractal property because we would obtain a 
 unique morphism of $P\TC_c$,
 
   \[\xymatrix{B^{0}_P\ar[r]^(.4){!_P}&Coend(B^{\bullet}_P)}\]
which express an action of $B^{0}_P$ on $End(A_{0,P})$ (see section
\ref{Some_standard_diagrams_of_the_omega-transformations}).
 
 At this stage it is fundamental to remark that if for each tree $t$, the coloured $\omega$-operad $B^{t}_P$ keeps the kind of contractibility which is involved 
 (strict contractibility with contractible units, or contractibility), 
 then $Coend(B^{\bullet}_P)$ is not only equipped with a composition system (see section 
 \ref{Composition_Systems}) 
 but also has the property $P$ (see proposition \ref{proposition-contractible} below).
 
 We didn't resolve yet the case $P=S_{u}$, $P=C$, but we believe that it is the case.
 For the rest of this article we accept the following hypotheses:
\begin{Hypothesis}
   \label{Hypothesis}      
  Each coloured $\omega$-operad $B^{t}_{P}$ ($t\in\mathbb{T}ree$) built in $\TC_c$ has the following virtue
 \begin{itemize}
   \item Each $B^{t}_{S_{u}}$ is strictly contractible with contractible units.  
   \item Each $B^{t}_{C}$ is contractible.
  \end{itemize}
\end{Hypothesis} 

We have no full proof of this hypothesis at the moment. For example consider the tree $t=1(1)\star_0 ^1(1)$. The corresponding operad $B^t_C$ is given by the pushout

\[\xymatrix{B^{0}_{C}\ar[d]_{\delta^{0}_{1}}
  \ar[rr]^{\kappa^{0}_{1}}&&B^{1}_{C}\ar[d]^{i_{1}}\\
  B^{1}_{C}\ar[rr]_{i_{2}}&&
    B^{1}_{C}\underset{B^{0}_{C}}\sqcup B^{1}_{C}}\] 
The contractibility hypothesis states that this $\omega$-operad is contractible. Let us show how to see the appearance of the contraction cells in a simple example. For the following example we will use the same notation for the symbol of functor $F_1$ (which is a $1$-cell of
$B^{1}_{C}\underset{B^{0}_{C}}\sqcup B^{1}_{C}$, but which can have arity $1(1)$ or $2(1)$, depending on where this symbol lives for each 
$B^{1}_{C}$ of this pushout). Consider
the $1$-cells $$x=\gamma_1(\gamma_1(F_{1};\gamma_1(\mu^{1}_{0};u_{1}\star^{1}_{0}\mu^{1}_{0}));
F_{1}\star^{1}_{0}F_{1}\star^{1}_{0}F_{1})$$ 
and 
$$y=\gamma_1(\gamma_1(F_{1};\gamma_1(\mu^{1}_{0};\mu^{1}_{0}\star^{1}_{0}u_{1}));
F_{1}\star^{1}_{0}F_{1}\star^{1}_{0}F_{1}).$$ 
It is not difficult to show that 
the arity of $x$ and $y$ is $1(1)\star^{1}_{0}1(1)\star^{1}_{0}1(1)$, and also that $x\|y$. The $2$-cell 
$$\gamma_2([\gamma_1(F_{1};\gamma_1(\mu^{1}_{0};u_{1}\star^{1}_{0}\mu^{1}_{0}));\gamma_1(F_{1};\gamma_1(\mu^{1}_{0};\mu^{1}_{0}\star^{1}_{0}u_{1}))];
1^{1}_{2}(F_{1})\star^{2}_{0}1^{1}_{2}(F_{1})\star^{2}_{0}1^{1}_{2}(F_{1}))$$
 is a coherence cell connecting $x$ and $y$. At this stage we can see that such coherence cells emerge from the contractibility of $B^{1}_{C}$ plus the operadical multiplication of $B^{1}_{C}\underset{B^{0}_{C}}\sqcup B^{1}_{C}$. 
Unfortunately, it is an impossible task to try to generalise these calculations because the combinatorics becomes unmanageable very quickly. 
We believe, however, that there is a more elegant  way to prove this hypothesis based on abstract homotopy theory. Indeed, according to our construction, the    
$\omega$-operad $B_C^t$ is a finite (wide) pushout of contractible  colored operads. If we assume that  there is a nice enough model structure on the category of colored $\omega$-operads such that
 the operads $B_C^n$ are contractible and cofibrant  in the model theoretic sense, and all morphisms in the pushout are cofibrations then the contractibility of $B^t_C$ will follow from the standard model theoretic argument. The existence of such a model structure is very  important not only for the contractibility hypothesis but for many other applications. It will be a subject of our future paper.  

\begin{remark}
 The reader can keep in mind a striking analogy between our operad $Coend(B^{\bullet}_C)$ and the operad $Coend(D^{\bullet})$ constructed by Michael Batanin in \cite{bat:monglob}. The coglobular complex $D^{\bullet}$ consists of the standard topological disks. An amazing fact is that $Coend(D^{\bullet})$ turns out to be a contractible $\omega$-operad which acts naturally on globular complex of points, paths, $2$-paths etc. of a topological space. In this way fundamental $\omega$-groupoid functor is constructed in \cite{bat:monglob}. We believe that our coglobular complex $B^{\bullet}_C$ plays the same role for the homotopy theory of $\omega$-operads, as $D^{\bullet}$ does for homotopy theory of  topological spaces.
\end{remark}

\begin{proposition}
\label{proposition-contractible}
Under the hypotheses above, $Coend(B^{\bullet}_{C})$ is contractible and $Coend(B^{\bullet}_{S_{u}})$ is strictly contractible with contractible
units.
\end{proposition}
\begin{proof}
   Consider two $(n-1)$-cells of $Coend(B^{\bullet}_{C})$ which are parallels and have the same arity. That means we give ourselves a diagram in $\TC_{c}$
 \[\xymatrix{B_{C}^{n-1}
    \ar[rr]<+2pt>^{f^{-}_{n-1}}\ar[rr]<-2pt>_{f^{+}_{n-1}}&&
       B_{C}^{t}\\    
       B_{C}^{n-2}\ar[u]<+2pt>^{\delta^{n-2}_{n-1}}\ar[u]<-2pt>_{\kappa^{n-2}_{n-1}}
       \ar[rr]<+2pt>^{f^{-}_{n-2}}\ar[rr]<-2pt>_{f^{+}_{n-2}}
       &&B_{C}^{\partial t}\ar[u]<+2pt>^{\delta^{n-2}_{n-1}}\ar[u]<-2pt>_{\kappa^{n-2}_{n-1}}\\
       B_{C}^{1}\ar@{.>}[u]<+2pt>^{}\ar@{.>}[u]<-2pt>_{}      
        \ar[rr]<+2pt>^{f^{-}_{1}}\ar[rr]<-2pt>_{f^{+}_{1}}&&
       B_{C}^{\partial^{k-1}t}\ar@{.>}[u]<+2pt>^{}\ar@{.>}[u]<-2pt>_{}\\    
       B_{C}^{0}\ar[u]<+2pt>^{\delta^{1}_{0}}\ar[u]<-2pt>_{\kappa^{1}_{0}}
       \ar[rr]<+2pt>^{f^{-}_{0}}\ar[rr]<-2pt>_{f^{+}_{0}}
       &&B_{C}^{0}\ar[u]<+2pt>^{\delta^{1}_{0}}\ar[u]<-2pt>_{\kappa^{1}_{0}}       
    }\]   
such that this diagram commute serially. In particular we have 
    
  \begin{align*}
& f^{-}_{n-1}\delta^{n-2}_{n-1}= f^{+}_{n-1}\delta^{n-2}_{n-1}\\
& \qquad \text{and}\\
& \qquad\qquad f^{-}_{n-1}\kappa^{n-2}_{n-1}= f^{+}_{n-1}\kappa^{n-2}_{n-1}
\end{align*}  

 Contemplate the diagram
      
    \[\xymatrix{B_{C}^{n-1}
    \ar[rrr]<+2pt>^{f^{-}_{n-1}}\ar[rrr]<-2pt>_{f^{+}_{n-1}}&&& B_{C}^{t}\\     
   B_{C}^{n-2}\ar[u]<+2pt>^{\delta^{n-2}_{n-1}}\ar[u]<-2pt>_{\kappa^{n-2}_{n-1}}&&& 
    C^{n-1}\ar[lllu]^{\eta^{n-1}}\ar[u]<+2pt>^{\overline{f^{-}}_{n-1}}\ar[u]<-2pt>_{\overline{f^{+}}_{n-1}}\\
    &&&C^{n-2}\ar[lllu]^{\eta^{n-2}}\ar[u]<+2pt>^{\delta^{n-2}_{n-1}}\ar[u]<-2pt>_{\kappa^{n-2}_{n-1}}}\] 
Call $\alpha^{n-1}$ the principal cell in $C^{n-1}$ and $\alpha^{n-2}$  the principal cell in $C^{n-2}$. By
  definition we have $\delta^{n-2}_{n-1}(\alpha^{n-2})=s^{n-1}_{n-2}(\alpha^{n-1})$ and
      $\kappa^{n-2}_{n-1}(\alpha^{n-2})=t^{n-1}_{n-2}(\alpha^{n-1})$.
      
 We have $\overline{f^{-}}_{n-1}(\alpha^{n-1})\| \overline{f^{+}}_{n-1}(\alpha^{n-1})$, because 
  
   \begin{align*}
   s^{n-1}_{n-2}\overline{f^{-}}_{n-1}(\alpha^{n-1}) &=\overline{f^{-}}_{n-1}(s^{n-1}_{n-2}(\alpha^{n-1}))\\
                                              &=\overline{f^{-}}_{n-1}(\delta^{n-2}_{n-1}(\alpha^{n-2}))\\
                                              &=f^{-}_{n-1}\delta^{n-2}_{n-1}\eta^{n-2}(\alpha^{n-2}) \\
                                              &=f^{+}_{n-1}\delta^{n-2}_{n-1}\eta^{n-2}(\alpha^{n-2})\\
                                              &=\overline{f^{+}}_{n-1}(\delta^{n-2}_{n-1}(\alpha^{n-2}))\\
                                              &=\overline{f^{+}}_{n-1}(s^{n-1}_{n-2}(\alpha^{n-1}))\\
                                              &=s^{n-1}_{n-2}\overline{f^{+}}_{n-1}(\alpha^{n-1})                                         
     \end{align*}      
 In the same way we can prove that
	   \[t^{n-1}_{n-2}\overline{f^{-}}_{n-1}(\alpha^{n-1})=t^{n-1}_{n-2}\overline{f^{+}}_{n-1}(\alpha^{n-1})\] 
 But $B_{C}^{t}$ is contractible thus $\overline{f^{-}}_{n-1}(\alpha^{n-1})$ and  $\overline{f^{+}}_{n-1}(\alpha^{n-1})$   
    are connected by the coherence cell
     \[[\overline{f^{-}}_{n-1}(\alpha^{n-1});\overline{f^{+}}_{n-1}(\alpha^{n-1})]\]
We then build  
    $\xymatrix{C^{n}\ar[r]^{\overline{f_n}}&B_{C}^{t}}$  as follow: If $\alpha^{n}$ is the principal cell of  
      $C^{n}$ then we put 
             \[\overline{f_n}(\alpha^{n})=[\overline{f^{-}}_{n-1}(\alpha^{n-1});\overline{f^{+}}_{n-1}(\alpha^{n-1})]\]    
Thus we obtain
     
     \[\xymatrix{B_{C}^{n}\ar[rrd]^{f_{n}}\\
    B_{C}^{n-1}\ar[u]<+2pt>^{\delta^{n}_{n-1}}\ar[u]<-2pt>_{\kappa^{n}_{n-1}}
    \ar[rr]<+2pt>^{f^{-}_{n-1}}\ar[rr]<-2pt>_{f^{+}_{n-1}}
        &&B_{C}^{t}}\]       
and finally we put $f_{n}:=[f^{-}_{n-1};f^{+}_{n-1}]$. 

The proof of the strict contractibility with contractible units of $Coend(B^{\bullet}_{S_{u}})$ is entirely similar.
\end{proof}

 \subsection{Composition Systems}
   \label{Composition_Systems} 
   
 $B^{\bullet}_{P}$, or $B^{\bullet}$ for short, denotes either the coglobular complex
 $B^{\bullet}_{S_{u}}$, or $B^{\bullet}_{C}$ in $\TC_c$. 
  Also denote by $B^{n}\underset{B^{p}}\sqcup B^{n}$
 the $3$-coloured $\omega$-operad in $\TC_c$ which is obtain by pushing out, in $\TC_c$, the following diagram

  \[\xymatrix{B^{p}\ar[d]_{\delta^{n}_{p}}\ar[r]^{\kappa^{n}_{p}}&B^{n}\\
  B^{n}}\] 
where $\delta^{p}_{n}=\delta^{n}_{n-1}...\delta^{p}_{p+1}$ and 
  $\kappa^{p}_{n}=\kappa^{n}_{n-1}...\kappa^{p}_{p+1}$.  
 For each integers $0\leqslant p<n$ we
 are going to define a morphism in $\TG_{p,c}$
 
  \[\xymatrix{C^{n}\ar[rrr]^(.4){\mu^{n}_{p}}&&&B^{n}\underset{B^{p}}\sqcup B^{n}}\]
which, depending on the universality property required, gives us a unique morphism in $\TC_c$, that
  we still call $\mu^{n}_{p}$ because there is no risk of confusion, 
  
  \[\xymatrix{B^{n}\ar[rrr]^(.4){\mu^{n}_{p}}&&&B^{n}\underset{B^{p}}\sqcup B^{n}}\]
For instance, if we accept the contractibility hypothesis \ref{Hypothesis} whose consequence is that $B^{n}_{C}\underset{B^{p}_{C}}\sqcup B^{n}_{C}$ 
  is still an object of  $C\TC_c$, the universal map $\xymatrix{C^{n}\ar[r]^{\eta^{n}_{C}}&B^{n}_{C}}$  
  gives us such morphism $\mu^{n}_{p}$. We have similar technology for $P=S_{u}$. 
  The key point to defining these morphisms $\mu^{n}_{p}$ is first to describe the different compositions
 $\circ^{n}_{p}$ of the strict higher transformations.   
 If $0< p<n$,  we know that for two strict $n$-transformations $\sigma$ and $\tau$, we have
 
  \[(\sigma\circ^n_p \tau)(a):=\sigma(a)\circ^{n-1}_{p-1}\tau(a)\]
whose operadic interpretation is given by the cell $\gamma(\mu^{n-1}_{p-1};\sigma\ast^{n-1}_{p-1}\tau)$. Then the morphism in $\TG_{p,c}$
          
     \[\xymatrix{C^{n}\ar[rr]^{\mu^{n}_{p}}&&B^{n}\underset{B^{p}}\sqcup B^{n}}\] 
sends the principal cell $\tau$ of $C^{n}$ to the $(n-1)$-cell $\gamma(\mu^{n-1}_{p-1};\sigma\ast^{n-1}_{p-1}\tau)$ 
 of $B^{n}\underset{B^{p}}\sqcup B^{n}$, sends for each $i\in \mathbb{N}$, the     
     $i$-cell $F_{i}$ of $C^{n}$ to the $i$-cell $F_{i}$ of $B^{n}\underset{B^{p}}\sqcup B^{n}$, and sends the 
  $i$-cell $G_{i}$ of $C^{n}$ to the $i$-cell $H_{i}$ of $B^{n}\underset{B^{p}}\sqcup B^{n}$. This morphism 
  of  $\TG_{p,c}$ is boundary preserving in an evident sense. 
 
If $p=0$ it is a bit more complex. We are in the situation of the pushout diagram below
   
 \[\xymatrix{B^{0}\ar[d]_{\delta^{0}_{n}}
  \ar[rr]^{\kappa^{0}_{n}}&&B^{n}\ar[d]^{i_{1}}\\
  B^{n}\ar[rr]_{i_{2}}&&
    B^{n}\underset{B^{0}}\sqcup B^{n}}\] 
First we describe the composition $\circ^{n}_{0}$ for the strict case, to be able to find the cells that we need in our $\omega$-operad.
 Consider the following diagram in the strict $\omega$-category of the strict $\omega$-categories. 
  \[\xymatrix{\mathcal{C}\rtwocell^{F}_{G}{\tau}&\mathcal{D}
   \rtwocell^{H}_{K}{\tau}&\mathcal{E}   \\}\]   
Here $\mathcal{C}$, $\mathcal{D}$ and $\mathcal{E}$ are $0$-cells (i.e strict $\omega$-categories), $F$, $G$, $H$ and $K$ are 
   $1$-cells (i.e strict $\omega$-functors) and $\tau$ and $\sigma$ are $n$-cells (i.e strict $n$-transformations).
 This picture describes $\tau$ and $\sigma$ with $2$-cells, but the reader must see them as $n$-cells. $\tau$ and $\sigma$ are such that : 
 $s^{n}_{0}(\sigma)=\mathcal{C}$, $t^{n}_{0}(\sigma)=s^{n}_{0}(\tau)=\mathcal{D}$, and $t^{n}_{0}(\tau)=\mathcal{E}$.  
 If $a\in \mathcal{C}(0)$, then $\xymatrix{F^{0}\ar[r]^{\tau(a)}&G^{0}}$ is an $(n-1)$-cells of $\mathcal{D}$ and
   it induces the following commutative square of $(n-1)$-cells in $\mathcal{E}$
   
   \[\xymatrix{H^{0}(F^{0}(a)) \ar[d]_{\sigma(F^{0})}\ar[rr]^{H^{n-1}(\tau(a))}&& 
   H^{0}(G^{0}(a))\ar[d]^{\sigma(G^{0})}\\
   K^{0}(F^{0}(a))\ar[rr]_{K^{n-1}(\tau(a))}&& K^{0}(G^{0}(a))}\]   
 which gives 
    \begin{align*}
    (\sigma\circ^{n}_{0}\tau)(a) &=\sigma(G_{0}(a))\circ^{n-1}_{0}H_{n-1}(\tau(a))\\
                                              &=K_{n-1}(\tau(a))\circ^{n-1}_{0}\sigma(F_{0}(a))
  \end{align*}    
and this gives the two principal $(n-1)$-cells of $B^{n}\underset{B^{0}}\sqcup B^{n}$ that we need:
  \begin{align*}
& \gamma^{n-1}(\mu^{n-1}_{0};\gamma(\sigma;G^{0})\ast^{n-1}_{0}\gamma(H^{n-1};\tau))\\
& \qquad \text{and}\\
& \qquad\qquad \gamma^{n-1}(\mu^{n-1}_{0};\gamma(K^{n-1};\tau)\ast^{n-1}_{0}\gamma(\sigma;F^{0}))
\end{align*}  
 Then we have two choices of 
     \[\xymatrix{C^{n}\ar[rr]^(.4){\mu^{n}_{o}}&&B^{n}\underset{B^{p}}\sqcup B^{n}}\] 
which send the principal cell $\tau$ of $C^{n}$ to      
   $\gamma^{n-1}(\mu^{n-1}_{0};\gamma(\sigma;G^{0})\ast^{n-1}_{0}\gamma(H^{n-1};\tau))$ or on   
   $\gamma^{n-1}(\mu^{n-1}_{0};\gamma(K^{n-1};\tau)\ast^{n-1}_{0}\gamma(\sigma;F^{0}))$, and for both
   cases, which send for each $i\in \mathbb{N}$, the     
     $i$-cell $F^{i}$ of $C^{n}$ to the $i$-cell $\gamma(F_{i};H_{i})$ of $B^{n}\underset{B^{0}}\sqcup B^{n}$, and the 
  $i$-cell $G^{i}$ of $C^{n}$ to the $i$-cell $\gamma(G_{i};K_{i})$ of $B^{n}\underset{B^{0}}\sqcup B^{n}$. These morphisms 
  of  $\TG_{p,c}$ are boundary preserving in an evident sense. 
  
  Now let us come back to the specific case of the coglobular complex $B^{\bullet}_{S_{u}}$, or $B^{\bullet}_{C}$ in $\TC_c$.
 Suppose we accept the fractality hypothesis (see section \ref{Hypothesis}) 
  for the $\omega$-operads $Coend(B^{\bullet}_{P})$, where $P$ can be either $S_{u}$, or $C$. 
  In that case, thanks to the universal property of $\eta^{n}$, we get the following unique morphisms of $\omega$-operads $\mu^{n}_{p}$
  and $\mu^{n}_{0}$ (the dotted arrows)
       
     \[\xymatrix{B^{n}\ar@{.>}[rr]^{\mu^{n}_{p}}&&B^{n}\underset{B^{p}}\sqcup B^{n}\\
     C^{n}\ar[u]^{\eta^{n}} \ar[rru]_{\mu^{n}_{p}} }\qquad\xymatrix{B^{n}\ar@{.>}[rr]^{\mu^{n}_{0}}&&B^{n}\underset{B^{0}}\sqcup B^{n}\\
     C^{n}\ar[u]^{\eta^{n}} \ar[rru]_{\mu^{n}_{0}} }\]  
With the identity morphisms of operads $\xymatrix{B^{n}\ar[r]^{1_{B^{n}}}&B^{n}}$ 
  
   \[\xymatrix@R-23pt{C^{0}\ar[r]^(.37){c_{w}}&Coend(B^{\bullet})\\
     \mu^{n}_{p}\ar@{|->}[r]& \mu^{n}_{p}\\
     u_{n}\ar@{|->}[r]& 1_{B^{n}}
     }\] 
we thus have the following conclusion :
\begin{proposition}
\label{proposition-system}
The $\omega$-operads of coendomorphisms $Coend(B^{\bullet}_{S_{u}})$ and $Coend(B^{\bullet}_{C})$ have composition systems.
\end{proposition}
  
\subsection{The strict $\omega$-category of strict $\omega$-categories} 
     \label{The_Strict_omega_category_of_the_Strict_omega_categories}   
     
 Consider the case $P=S_{u}$ ("Strict with contractible units"), i.e we deal with the category $S_{u}\TC_c$ of strict $\omega$-operads 
  with contractible units (see section \ref{Weak_contractibility_and_strict_contractibility}). The coglobular complex $B^{\bullet}_{S_{u}}$ of the section
  \ref{Contractibility_Hypothesis} produces the following globular complex in $\C$
      
    \[\xymatrix{\ar@{.>}[r]<+2pt>^{}\ar@{.>}[r]<-2pt>_{}
  &B_{S_{u}}^{n}\text{-}\mathbb{A}lg\ar[r]<+2pt>^{\sigma^{n}_{n-1}}\ar[r]<-2pt>_{\beta^{n}_{n-1}}
  &B_{S_{u}}^{n-1}\text{-}\mathbb{A}lg\ar@{.>}[r]<+2pt>^{}\ar@{.>}[r]<-2pt>_{}
  &B_{S_{u}}^{1}\text{-}\mathbb{A}lg\ar[r]<+2pt>^{\sigma^{1}_{0}}\ar[r]<-2pt>_{\beta^{1}_{0}}
  &B_{S_{u}}^{0}\text{-}\mathbb{A}lg }\]
 Also it is possible to prove the following proposition :
\begin{proposition}
    Objects of $B_{S_{u}}^{1}\text{-}\mathbb{A}lg$ are strict $\omega$-functors, and for each integer $n\geqslant 2$,
   objects of $B_{S_{u}}^{n}\text{-}\mathbb{A}lg$ are strict $n$-transformations.
\end{proposition}
The standard action of the coglobular complex $B^{\bullet}_{S_{u}}$ is given by the following diagram in $\TC_1$
   
   \[\xymatrix{Coend(B^{\bullet}_{S_{u}})
   \ar[rr]^{Coend(\mathbb{A}lg(.))}&& 
 Coend(A_{S_{u}}^{op})\ar[rr]^{Coend(Ob(.))}&&End(A_{0,S_{u}})}\]
It is an other specific standard action of the higher transformations. 
The monochromatic $\omega$-operad of coendomorphism $Coend(B^{\bullet}_{S_{u}})$ plays a central role
 for strict $\omega$-categories. We call it the \textit{indigo operad}\footnote{We use Newton's $7$ primary colours
of the Rainbow to denote $4$ relevant $\omega$-operads of coendomorphism of this article. We don't mention in this article the
\textit{red operad}, the \textit{orange operad}, and the \textit{blue operad}, which are respectivally specific for 
$\omega$-graphs,
reflexive $\omega$-graphs, and semi-strict $\omega$-categories, because they are very similar to others monochromatic $\omega$-operad of coendomorphism
of this article, and we don't need them to reached the main idea of this article.}. According to the hypotheses \ref{Hypothesis}, the indigo operad has 
 a composition system (see the proposition \ref{proposition-system}) and is strictly contractible with contractible units. 
 Thus we have a unique morphism in $\TC_{1}$
  
 \[\xymatrix{ B^{0}_{S_{u}}\ar[rr]^(.4){!_{s}}&&Coend(B^{\bullet}_{S_{u}}) }\] 
and we obtain a morphism of $\omega$-operads
 
      \[\xymatrix{ B^{0}_{S_{u}}\ar[rrr]^(.4){\mathfrak{S}_{u}}&&&End(A_{0,S_{u}})}\] 
which expresses an action of the $\omega$-operad  $B^{0}_{S_{u}}$ of 
 strict $\omega$-categories on the globular complex 
 $B^{\bullet}_{S_{u}}$-$\mathbb{A}lg(0)$ in $SET$
 of strict higher transformations, and thus gives a
    strict $\omega$-category structure on strict higher transformations.
    
\subsection{The weak $\omega$-category of weak $\omega$-categories} 
   \label{The_Weak_omega_category_of_the_Weak_omega_categories}

 Consider the case $P=C$ ("Contractible"), i.e we deal with the category $C\TC$ of contractible $\omega$-operads 
 (see \ref{Weak_contractibility_and_strict_contractibility}). The coglobular complex $B^{\bullet}_{C}$ of section
  \ref{Contractibility_Hypothesis} produces the following globular complex in $\C$
      
    \[\xymatrix{\ar@{.>}[r]<+2pt>^{}\ar@{.>}[r]<-2pt>_{}
  &B_{C}^{n}\text{-}\mathbb{A}lg\ar[r]<+2pt>^{\sigma^{n}_{n-1}}\ar[r]<-2pt>_{\beta^{n}_{n-1}}
  &B_{C}^{n-1}\text{-}\mathbb{A}lg\ar@{.>}[r]<+2pt>^{}\ar@{.>}[r]<-2pt>_{}
  &B_{C}^{1}\text{-}\mathbb{A}lg\ar[r]<+2pt>^{\sigma^{1}_{0}}\ar[r]<-2pt>_{\beta^{1}_{0}}
  &B_{C}^{0}\text{-}\mathbb{A}lg }\]      
and in the article \cite{kach:nscellsfinal} it was proved, with the old notion of contractibility, that we have the proposition
\begin{proposition} 
Algebras of dimension $2$ of $B_{C}^{1}\text{-}\mathbb{A}lg$ are pseudo-$2$-functors, and 
algebras of dimension $2$ of $B_{C}^{2}\text{-}\mathbb{A}lg$ are pseudo-$2$-natural transformations.
\end{proposition}  
However, with our new notion of contractibility (see section \ref{Weak_contractibility_and_strict_contractibility}), this proposition remains true, and the 
proof is exactly the same as the proof in the article \cite{kach:nscellsfinal}. The standard action of the coglobular complex $B^{\bullet}_{C}$ is 
given by the following diagram in $\TC_1$
 
   \[\xymatrix{Coend(B^{\bullet}_{C})
   \ar[rr]^{Coend(\mathbb{A}lg(.))}&& 
 Coend(A_{S}^{op})\ar[rr]^{Coend(Ob(.))}&&End(A_{0,C})}\]
It is an other specific standard action of the higher transformations. The monochromatic $\omega$-operad of coendomorphism 
$Coend(B^{\bullet}_{C})$ plays a central role for weak $\omega$-categories. We call it the \textit{violet operad}. Batanin's 
$\omega$-operad $B^{0}_{C}$ of weak $\omega$-categories is initial among contractible $\omega$-operads which 
 have a composition system. According to the contractibility hypothesis \ref{Hypothesis}, the violet operad has a composition system 
 (see proposition \ref{proposition-system}) and is contractible. Thus we have a unique morphism in $\TC_{1}$
  
 \[\xymatrix{ B^{0}_{C}\ar[rr]^{!_{C}}&&Coend(B^{\bullet}_{C}) }\] 
 and we obtain a morphism of $\omega$-operads
 
      \[\xymatrix{ B^{0}_{C}\ar[rrr]^(.4){\mathfrak{C}}&&&End(A_{0,C})}\] 
which expresses an action of the $\omega$-operad  $B^{0}_{C}$ of weak $\omega$-categories on the globular complex 
 $B^{\bullet}_{C}$-$\mathbb{A}lg(0)$ in $SET$ of weak higher transformations, and thus gives a
  weak $\omega$-category structure on the weak higher transformations. It is not difficult to prove that under this weak  $\omega$-category structure
  on $B^{\bullet}_{C}$-$\mathbb{A}lg(0)$, the composition of weak  $\omega$-functors is associative up to weak natural $\omega$-transformations. 
      
\section{Examples of $\omega$-operads with fractal property} 
\label{Examples_of_Standard_Actions}
     
  Consider now the case $P=Id$ ("Magmatic"), i.e we deal with the category $\TC_c$ of $\omega$-operads (see section\ref{T-graphs_with_contractible_units}).
  We apply the free functor (see section \ref{T-graphs_with_contractible_units})
    
   \[\xymatrix{\TG_{p,c}\ar[rrr]^{M}&&& \TC_c}\]      
to the coglobular complex of the higher transformations $C^{\bullet}$ in $\TG_{p,c}$
and we obtain a coglobular complex $B^{\bullet}_{Id}$ of $\omega$-operads in $\TC_{c}$
    
   \[\xymatrix{B_{Id}^{0}\ar[rr]<+2pt>^{\delta^{1}_{0}}\ar[rr]<-2pt>_{\kappa^{1}_{0}}
  &&B_{Id}^{1}\ar[rr]<+2pt>^{\delta^{1}_{2}}\ar[rr]<-2pt>_{\kappa^{2}_{1}}
  &&B_{Id}^{2}\ar@{.>}[r]<+2pt>^{}\ar@{.>}[r]<-2pt>_{}
  &B_{Id}^{n-1}\ar[rr]<+2pt>^{\delta^{n}_{n-1}}\ar[rr]<-2pt>_{\kappa^{n}_{n-1}}
  &&B_{Id}^{n}\ar@{.}[r]<+2pt>\ar@{.}[r]<-2pt>&}\]    
which produces the following globular complex in $\mathbb{C}AT$.    
    
    \[\xymatrix{\ar@{.>}[r]<+2pt>^{}\ar@{.>}[r]<-2pt>_{}
  &B_{Id}^{n}\text{-}\mathbb{A}lg\ar[r]<+2pt>^{\sigma^{n}_{n-1}}\ar[r]<-2pt>_{\beta^{n}_{n-1}}
  &B_{Id}^{n-1}\text{-}\mathbb{A}lg\ar@{.>}[r]<+2pt>^{}\ar@{.>}[r]<-2pt>_{}
  &B_{Id}^{1}\text{-}\mathbb{A}lg\ar[r]<+2pt>^{\sigma^{1}_{0}}\ar[r]<-2pt>_{\beta^{1}_{0}}
  &B_{Id}^{0}\text{-}\mathbb{A}lg }\]    
In particular $B^{0}_{Id}$ is the $\omega$-operad for $\omega$-magmas (see \cite{kach:infn}).
The standard action associated to $B^{\bullet}_{Id}$ is given by the following diagram in $\TC_1$
 
   \[\xymatrix{Coend(B^{\bullet}_{Id})
   \ar[rr]^{Coend(\mathbb{A}lg(.))}&& 
 Coend(A_{Id}^{op})\ar[rr]^{Coend(Ob(.))}&&End(A_{0,Id}) }\]
that we call the standard action of $\omega$-magmas, thus which is a specific
 standard action of higher transformations. The monochromatic $\omega$-operad $Coend(B^{\bullet}_{Id})$ of coendomorphism plays a central role
 for $\omega$-magmas. We call it the \textit{yellow operad}. Also we have the following proposition
\begin{proposition}
\label{proposition-fractality-BM}
$B^{0}_{Id}$ has the fractal property.
\end{proposition}
 If we compose the morphism $!_{Id}$ 
 \[\xymatrix{ B^{0}_{Id}\ar[rr]^(.4){!_{Id}}&&Coend(B^{\bullet}_{Id}) }\] 
with the standard action associated to $B^{\bullet}_{Id}$ we obtain a morphism of $\omega$-operads
  
      \[\xymatrix{ B^{0}_{Id}\ar[rrr]^(.4){\mathfrak{I}d}&&&End(A_{0,Id}) }\] 
which expresses an action of the $\omega$-operad  $B^{0}_{Id}$ of the 
 $\omega$-magmas on the globular complex 
 $B^{\bullet}_{Id}$-$\mathbb{A}lg(0)$ in $SET$
  of the $(n,\omega)$-magmas ($n\in\mathbb{N}$), and thus gives an 
    $\omega$-magma structure on the $(n,\omega)$-magmas ($n\in\mathbb{N}$). 
     
 Consider the case $P=Id_{u}$ ("Magmatic with contractible units"), i.e we deal with the category $Id_{u}\TC_c$ of $\omega$-operads with contactible units 
  (see section \ref{T-graphs_with_contractible_units}).
  We apply the free functor (see section \ref{T-graphs_with_contractible_units})
    
   \[\xymatrix{\TG_{p,c}\ar[rrr]^{Id_{u}}&&& Id_{u}\TC_c}\]      
to the coglobular complex of higher transformations $C^{\bullet}$ in $\TG_{p,c}$ 
and we obtain a coglobular complex $B^{\bullet}_{Id_{u}}$ of $\omega$-operads in $\TC_{c}$
    
   \[\xymatrix{B_{Id_{u}}^{0}\ar[rr]<+2pt>^{\delta^{1}_{0}}\ar[rr]<-2pt>_{\kappa^{1}_{0}}
  &&B_{Id_{u}}^{1}\ar[rr]<+2pt>^{\delta^{1}_{2}}\ar[rr]<-2pt>_{\kappa^{2}_{1}}
  &&B_{Id_{u}}^{2}\ar@{.>}[r]<+2pt>^{}\ar@{.>}[r]<-2pt>_{}
  &B_{Id_{u}}^{n-1}\ar[rr]<+2pt>^{\delta^{n}_{n-1}}\ar[rr]<-2pt>_{\kappa^{n}_{n-1}}
  &&B_{Id_{u}}^{n}\ar@{.}[r]<+2pt>\ar@{.}[r]<-2pt>&}\]    
which produces the following globular complex in $\mathbb{C}AT$.    
    
    \[\xymatrix{\ar@{.>}[r]<+2pt>^{}\ar@{.>}[r]<-2pt>_{}
  &B_{Id_{u}}^{n}\text{-}\mathbb{A}lg\ar[r]<+2pt>^{\sigma^{n}_{n-1}}\ar[r]<-2pt>_{\beta^{n}_{n-1}}
  &B_{Id_{u}}^{n-1}\text{-}\mathbb{A}lg\ar@{.>}[r]<+2pt>^{}\ar@{.>}[r]<-2pt>_{}
  &B_{Id_{u}}^{1}\text{-}\mathbb{A}lg\ar[r]<+2pt>^{\sigma^{1}_{0}}\ar[r]<-2pt>_{\beta^{1}_{0}}
  &B_{Id_{u}}^{0}\text{-}\mathbb{A}lg }\]   
 In particular $B^{0}_{Id_{u}}$ is the $\omega$-operad for reflexive $\omega$-magmas (see \cite{kach:infn}).
 The standard action associated to $B^{\bullet}_{Id_{u}}$ is given by the following diagram in $\TC_1$.
   
      \[\xymatrix{Coend(B^{\bullet}_{Id_{u}})
   \ar[rr]^{Coend(\mathbb{A}lg(.))}&& 
 Coend(A_{Id_{u}}^{op})\ar[rr]^{Coend(Ob(.))}&&End(A_{0,Id_{u}})}\]
It is a specific standard action of the higher transformations. The monochromatic $\omega$-operad $Coend(B^{\bullet}_{Id_{u}})$ 
of coendomorphism plays a central role for reflexive $\omega$-magmas. We call it the \textit{green operad}. Also we have the following proposition
\begin{proposition}
$B^{0}_{Id_{u}}$ has the fractal property.
\end{proposition}
If we compose the morphism $!_{Id_{u}}$ 
\[\xymatrix{ B^{0}_{Id_{u}}\ar[rr]^(.4){!_{Id_{u}}}&&Coend(B^{\bullet}_{Id_{u}}) }\] 
with the standard action associated to $B^{\bullet}_{Id_{u}}$ we obtain a morphism of $\omega$-operads

   \[\xymatrix{B^{0}_{Id_{u}}\ar[rrr]^(.4){\mathfrak{C}_u}&&&End(A_{0,Id_{u}})}\] 
which expresses an action of the $\omega$-operad  $B^{0}_{Id_{u}}$ of reflexive
 $\omega$-magmas on the globular complex 
 $B^{\bullet}_{Id_{u}}$-$\mathbb{A}lg(0)$ in $SET$
  of the reflexive $(n,\omega)$-magmas ($n\in\mathbb{N}$), and thus gives a 
    reflexive $\omega$-magma structure on the reflexive $(n,\omega)$-magmas ($n\in\mathbb{N}$).

\bigbreak{}
  \begin{minipage}{1.0\linewidth}
    Camell \textsc{Kachour}\\
    Macquarie University, Department of Mathematics\\
    Phone: 00612 9850 8942\\
    Email:\href{mailto:camell.kachour@gmail.com}{\url{camell.kachour@gmail.com}}
  \end{minipage}
  

\vspace{1cm}  

\begin{thebibliography}{000}

\bibitem{Arathese} Dimitri Ara, \textit{Sur les $\infty$-groupo{\"i}des de Grothendieck et une 
variante $\infty$-cat{\'e}gorique}, \url{http://arxiv.org/pdf/math/0607820v2} (2010).\label{Arathese}
\bibitem{bat:monglob} Michael Batanin, \textit{Monoidal Globular Categories As a Natural Environment for
the Theory of Weak-$n$-Categories}, Advances in Mathematics (1998), volume 136, pages 39--103.
\label{bat:monglob}
\bibitem{batross:multi} Michael Batanin and Ross Street, \textit{The universal property of the multitude of trees}, Journal of Pure and Applied Algebra (2000), volume 154, pages 3--13.
\label{batross:multi}
\bibitem{batanin02:_penon_method_of_weaken_algeb_struc} Michael Batanin, 
\textit{On the Penon method of weakening algebraic structures}, Journal of Pure and Applied Algebra (2002), volume 172, pages 1--23.
\label{batanin02:_penon_method_of_weaken_algeb_struc}
\bibitem{bat:eckmann} Michael Batanin, \textit{The Eckmann-Hilton argument and higher operads}, Advances in Mathematics (2008), volume 217, pages 334--385.\label{bat:eckmann}
\bibitem{Ber00} Clemens Berger, \textit{A cellular nerve for higher categories}, 
Advances in Mathematics (2002), volume 169, pages 118--175.\label{Ber00}
\bibitem{francisborceux:handbook2} Francis Borceux, \textit{Handbook of Categorical Algebra}, Cambridge University Press (1994), volume 2.\label{francisborceux:handbook2}
\bibitem{burroni-tcat} Albert Burroni, \textit{T-cat{\'e}gories (cat{\'e}gories dans un triple)}, 
Cahiers de Topologie et de G{\'e}om{\'e}trie
		  Diff{\'e}rentielle Cat{\'e}gorique (1971), volume 12, pages 215--321.\label{burroni-tcat}
\bibitem{Cisinsk:Bat} Denis-Charles Cisinski, \textit{Batanin higher groupoids and homotopy types}, Contemporary Mathematics (2007), volume 431, pages 171--186.\label{Cisinsk:Bat}
\bibitem{GarnerHomoCat} Richard Garner, \textit{Homomorphisms of higher categories}, Advances in Mathematics (2010), volume 224, pages 2269--2311.\label{GarnerHomoCat}
\bibitem{grothendieck83:_pursuin_stack} Alexander Grothendieck, \textit{Pursuing Stacks}, Typed manuscript (1983).\label{grothendieck83:_pursuin_stack}
\bibitem{JoyalTierney} Andr{\'e} Joyal and Myles Tierney, \textit{Quasi-categories vs Segal spaces}, 
Categories in algebra, geometry and mathematical physics, 277--326, Contemporary Mathematics. 431, Amer. Math. Soc., Providence, RI, 2007.\label{JoyalTierney}
\bibitem{joyal:theta} Andr{\'e} Joyal, \textit{Disks, duality and $\Theta$-categories}, Preprint (1997).
\label{joyal:theta}
\bibitem{kach3:redmacq} Camell Kachour, \textit{Toward the Operadical definition of the Weak Omega Category of the Weak Omega Categories, Part 3: The Red Operad}, Australian Category Seminar, Macquarie University (2010).\label{kach3:redmacq}
\bibitem{kamelkachour:defalg} Kamel Kachour, \textit{D{\'e}finition alg{\'e}brique des cellules non-strictes}, Cahiers de Topologie et de G{\'e}om{\'e}trie Diff{\'e}rentielle Cat{\'e}gorique (2008), volume 1, 
pages 1--68.\label{kamelkachour:defalg}
\bibitem{kach:nscellsfinal} Camell Kachour, \textit{Operadic Definition of the Non-strict Cells}, 
Cahiers de Topologie et de G{\'e}om{\'e}trie Diff{\'e}rentielle Cat{\'e}gorique (2011), volume 4, 
pages 1--48.\label{kach:nscellsfinal}
\bibitem{kach:infn} Camell Kachour, \textit{Algebraic Definition of weak $(\infty,n)$-Categories}, Submitted in Advances in Mathematics (2012).\label{code}
\bibitem{lein1:oper} Tom Leinster, \textit{Higher Operads, Higher Categories}, London Mathematical Society Lecture Note Series, Cambridge University Press (2004), volume 298.\label{lein1:oper}
\bibitem{LurieTopos} Jacob Lurie, \textit{Higher topos theory}, Princeton University Press (2009).\label{LurieTopos}
\bibitem{makkai} Michael Makkai, \textit{The multitopic omega-category of all multitopic omega-categories}, \url{http://www.math.mcgill.ca/makkai/mltomcat04/mltomcat04.pdf} (2004).\label{makkai}
\bibitem{malts:group} George Maltsiniotis, \textit{Infini groupo{\"\i}des non strictes, d'apr{\`e}s Grothendieck}, 
\url{http://www.math.jussieu.fr/~maltsin/ps/infgrart.pdf} (2007).\label{malts:group}
\bibitem{penon1999} Jacques Penon, \textit{Approche polygraphique des $\infty$-cat{\'e}gories non-strictes}, Cahiers de Topologie et de G{\'e}om{\'e}trie Diff{\'e}rentielle Cat{\'e}gorique (1999), volume 1, 
pages 31--80.\label{penon1999}
\bibitem{RezkCartesian} Charles Rezk, \textit{A cartesian presentation of weak $n$-categories}, Geometry and Topology (2010), volume 14.\label{RezkCartesian}
\bibitem{SimpsonHomoCat} Carlos Simpson, \textit{Homotopy Theory of Higher Categories. From Segal Categories to $n$-Categories and Beyond}, New Mathematical Monographs,  
 Cambridge University Press (2011), volume 19.\label{SimpsonHomoCat}
 \bibitem{street-petit-topos} Ross Street, \textit{The petit topos of Globular sets}, Journal of Pure and Applied Algebra (2000), volume 154, pages 299--315.\label{street-petit-topos}
\bibitem{weber:pseudo} Mark Weber, \textit{Operads within monoidal pseudo algebras}, Applied Categorical Structures (2005), volume 13, pages 389--420.\label{weber:pseudo}
\bibitem{mark-topos} Mark Weber, \textit{Yoneda structures from $2$-toposes}, Applied Categorical Structures (2007), volume 15, pages 259--323.\label{mark-topos}
\end{thebibliography}
\end{document}